\documentclass[11pt]{amsart}

\usepackage[T1]{fontenc}
\usepackage{amsmath,amssymb,amsthm}
\usepackage{tikz-cd}
\usepackage{microtype}
\usepackage[hidelinks]{hyperref}

\newtheorem{theorem}{Theorem}[section]
\newtheorem{lemma}[theorem]{Lemma}
\newtheorem{corollary}[theorem]{Corollary}

\theoremstyle{definition}
\newtheorem{definition}[theorem]{Definition}
\newtheorem{assumption}[theorem]{Assumption}

\theoremstyle{remark}
\newtheorem*{remark}{Remark}

\title{Donaldson-Sun Theory in the Conic Case}
\author{Arka Karmakar}
\date{\today}

\begin{document}

\maketitle

\begin{abstract}

We extend the Donaldson--Sun theory of metric tangent cones to non-collapsing
Gromov--Hausdorff limits of conical K\"ahler--Einstein pairs whose boundary
coefficients lie in a fixed finite subset of $\mathbb Q$, and prove uniqueness of the log metric tangent cone. Furthermore, under a mild lc compatibility condition, we relate this with the Li-Xu and Li-Liu-Xu stable degeneration machinery. 

We also construct polarized smooth Kähler metrics on $\mathbf{CP}^2$ with a uniform lower Ricci bound and volume non-collapsing, whose limit has nonunique tangent cones, showing that the Kähler-Einstein assumption is crucial for rigidity. \end{abstract}

\section{Introduction}

The study of Gromov-Hausdorff limits of K\"ahler-Einstein manifolds is a
central meeting point of metric geometry and K-stability. Cheeger-Colding
theory describes the general structure of non-collapsed Ricci-limit spaces,
while Donaldson-Sun show that K\"ahler-Einstein limits carry substantially
more algebraic structure: the limit is a normal projective variety, its metric
tangent cones are normal affine varieties, and every tangent cone is obtained
from the local germ by a two-step algebraic degeneration
\cite{DSI,DSII}. More precisely, if $p\in X_\infty$ and $\nu$ is the metric
order on $\mathcal O_{X_\infty,p}$, then $ W=\operatorname{Spec}\operatorname{gr}_\nu \mathcal O_{X_\infty,p}$ is an intermediate affine cone, and a metric tangent cone $C(Y)$ is obtained
from $W$ by a further equivariant degeneration.

From the algebro-geometric point of view, Donaldson--Sun theory realizes it by
an algebraic degeneration of the local germ. In the conic setting one must, in
addition, recover the boundary divisor and follow it through this degeneration.

Conical K\"ahler-Einstein metrics play a basic role in the proof of the
Yau-Tian-Donaldson conjecture through the work of
Chen-Donaldson-Sun \cite{CDSI,CDSII,CDSIII}. Their compactness theory,
however, does not by itself give a logarithmic version of the
Donaldson-Sun tangent-cone theorem. The new difficulty is that we need sharper handling of the codimension $4$ singularities that are formed in the Gromov--Hausdorff limits. 

Our goal is to establish the corresponding logarithmic tangent-cone picture,
recovering the boundary from the metric and following it through the algebraic
degenerations leading to the tangent cone.

Under the hypotheses stated in Section~2, our main result is the following.

\begin{theorem}[Logarithmic Donaldson-Sun theorem]
\label{thm:main-log-ds}
Let $(X_i,\Delta_i,\omega_i) \longrightarrow (X_\infty,\omega_\infty)$
be a non-collapsing Gromov--Hausdorff limit of conical
K\"ahler--Einstein pairs of complex dimension $n$, with coefficients in a
fixed finite subset of $\mathbb Q$ and satisfying the uniform assumptions of
Section~2. Then the following hold.

\begin{enumerate}
\item The metric determines an effective rational Weil divisor $ \Delta_\infty=\sum_a(1-\beta_a)D_a$ on the normal projective variety $X_\infty$. The divisor
$K_{X_\infty}+\Delta_\infty$ is $\mathbb Q$-Cartier, and
$\omega_\infty$ satisfies the weak conical K\"ahler--Einstein equation on
$(X_\infty,\Delta_\infty)$.

\item If $(X_\infty,r_j^{-2}\omega_\infty,p) \longrightarrow (C(Y),\omega_C,o)$
is a metric tangent cone, then $C(Y)$ is a normal affine variety carrying a
canonical effective rational Weil divisor $\Delta_C$. The pair
$(C(Y),\Delta_C)$ is log $\mathbb Q$-Gorenstein, and $\omega_C$ is a weak
conical Ricci-flat K\"ahler cone metric on
$(C(Y),\Delta_C,\xi)$.

\item The metric filtration of $\mathcal O_{X_\infty,p}$ is finitely
generated. If  $W=\operatorname{Spec}\operatorname{gr}_\nu \mathcal O_{X_\infty,p} $
then the boundary has an associated graded cycle $\Delta_W$, and there is a
logarithmic two-step degeneration
$ (X_\infty,\Delta_\infty,p)
        \rightsquigarrow
        (W,\Delta_W,\xi)
        \rightsquigarrow
        (C(Y),\Delta_C,\xi).
$
The boundary cycle produced by the second algebraic degeneration agrees, with
generic multiplicities, with the metric blow-up of $\Delta_\infty$.

\item The log affine cone $(C(Y),\Delta_C,\xi)$ is independent of the
subsequence defining the tangent cone. Consequently the metric tangent cone at
$p$ is unique.
\end{enumerate}
\end{theorem}

In particular, it is not a priori clear from the metric construction alone that a
tangent cone arises from an algebraic degeneration of the local germ, let alone
that such a tangent cone is unique. Polarization, volume noncollapsing, and a uniform lower
Ricci bound alone do not imply uniqueness of metric tangent cones.  The
following result is proved in Appendix~\ref{app:nonunique-tangent-cones}.

\begin{theorem}[Failure of tangent-cone uniqueness under a lower Ricci bound]
There exists a sequence of smooth polarized K\"ahler surfaces, with
polarization in one fixed integral K\"ahler class, a uniform Ricci lower
bound, and uniform noncollapse, whose Gromov--Hausdorff limit has a point
with nonunique metric tangent cones.
\end{theorem}

Coming back to the proof of ~\ref{thm:main-log-ds}, The first issue is to identify which parts of the Cheeger--Colding singular set
actually come from the conical boundary. The finite set of cone angles gives a
uniform gap in volume density, so the metric regular set is a fixed
$\epsilon$-regular set. The higher stratum $S_4$ has real codimension at least
four, while at points of $S_2\setminus S_4$ the Liu-Székelyhidi charts,
together with a transverse holonomy calculation, trap the approximating
divisors near the axis of a cone
$\mathbb C_\beta\times\mathbb C^{n-1}$ and give a uniform bound on their
transverse degree. This separates the divisorial conical locus from the higher
metric singularities and produces the boundary $\Delta_\infty$. The same
analysis under blow-up gives $\Delta_C$ on a tangent cone. A point that will be
important later is that the rescaled boundaries converge as weighted analytic
cycles.

The codimension-two cutoff issue is already present in the single-divisor conic
setting of Chen-Donaldson-Sun \cite[Sections~2.1--2.3]{CDSII}, where tangent
cones are shown to be good in the capacity sense. In our more general log-pair
setting, we use the local description of $S_2\setminus S_4$ to obtain the same
cutoff input for the Donaldson--Sun $\bar\partial$/grafting argument.
This gives rigidity of the holomorphic spectrum on the
space of tangent cones and hence the metric filtration and the intermediate
cone $W=\operatorname{Spec}(\operatorname{gr}_\nu(
\mathcal O_{X_\infty,p}))$.

The other new point is to follow the boundary through this filtration. Since
$\Delta_\infty$ is only a Weil divisor and its components may merge under
degeneration, ambient Hilbert convergence alone does not determine the limiting
boundary. We track the integral divisorial data associated with
$N\Delta_\infty$ through the adapted bases. Their convergence under the Donaldson--Sun degeneration, followed by Poincar\'e--Lelong (or a transverse Rouch\'e argument), identifies the algebraic boundary on the tangent cone with the metric blow-up boundary
$N\Delta_C$.

Once this comparison is established, the final uniqueness argument is close to
Donaldson--Sun: the possible log tangent cones lie in a common Hilbert--Chow
incidence space, and the Luna-slice argument is applied to the pair rather than
only to the underlying affine cone. This shows that the log affine tangent cone
is unique, and uniqueness of the conical Ricci-flat metric then gives uniqueness
of the metric tangent cone.

There is a natural further comparison with normalized-volume stable
degeneration. Koll\'ar's KDiv space provides a suitable parameter space for divisors in a family. The remaining issue is the log-canonical sheaf. To obtain a genuine
$\mathbb Q$-Gorenstein degeneration one must know that the relevant reflexive
log-pluricanonical sheaves commute with taking the associated graded object.
Koll\'ar's hull theory gives locally closed strata on which this compatibility
holds, but we do not prove that the specific Donaldson--Sun degenerations remain
in such a stratum. We therefore isolate this compatibility as an additional
assumption in the final section. Under it, the construction above agrees with
the Li--Xu and Li--Wang--Xu stable-degeneration picture.
\section*{Acknowledgements}

The author thanks his advisor Prof. Gábor Székelyhidi for his guidance and support, and also for suggesting this topic. The author also thanks
Colin Fan and Prof. Yuchen Liu for helpful discussions on algebraic geometry; especially the section on special test configurations. The author also acknowledges the use of GPT-5.6 Sol for proofreading and for assistance in developing the construction in Appendix~\ref{app:nonunique-tangent-cones}. The author takes full responsibility for all mathematical claims and any errors.

\section{Preliminaries}

Throughout the paper we use the log-Fano normalization
\[
        [\omega]=c_1\bigl(-(K_X+\Delta)\bigr),
\]
so that on the complement of the boundary the Einstein equation is
$\operatorname{Ric}(\omega)=\omega$.

\begin{definition}[Conical K\"ahler--Einstein pair]
Let $X$ be a smooth compact K\"ahler manifold and let
$D=\bigcup_a D_a$ be a finite union of divisors. Fix
$\beta_a\in(0,1)$ and set
\[
        \Delta=\sum_a(1-\beta_a)D_a.
\]
Assume that $-(K_X+\Delta)$ is a K\"ahler $\mathbb Q$-class.

A \emph{conical K\"ahler--Einstein metric} on $(X,\Delta)$ is a closed
positive $(1,1)$-current $\omega$ with locally bounded potentials, smooth and
K\"ahler on $X\setminus D$, for which there are
\[
        \omega_0\in c_1\bigl(-(K_X+\Delta)\bigr),
        \qquad
        \phi\in\operatorname{PSH}(X,\omega_0)\cap L^\infty(X),
\]
with $\omega=\omega_0+\sqrt{-1}\partial\bar\partial\phi$, and such that on
$X\setminus D$,
\[
        \omega^n
        =
        \omega_0^n e^{-\phi+F}
        \prod_a |s_a|_{h_a}^{2\beta_a-2}.
\]
Here $s_a$ is a defining section of $D_a$, $h_a$ is a Hermitian metric on
$\mathcal O_X(D_a)$. If $|e|_h^2=e^{-\psi}$ in a local holomorphic frame, we write
$\Theta(h):=\sqrt{-1}\,\partial\bar\partial\psi$ for the Chern curvature of $h$.
$F$ is locally bounded and satisfies
\[
        \sqrt{-1}\partial\bar\partial F
        =
        \operatorname{Ric}(\omega_0)-\omega_0
        -\sum_a(1-\beta_a)\Theta(h_a)
\]
on $X\setminus D$.
\end{definition}

For this paper, we consider a family $\mathcal F$ of conical K\"ahler--Einstein
pairs $(X_i,\Delta_i,\omega_i)$, where $\Delta_i=\sum_a (1-\beta_{i,a})D_{i,a}$
satisfying the following assumptions:
\begin{itemize}
    \item There exists an integer $N_{\mathcal F}$ such that
    $\beta_{i,a}\in \frac{1}{N_{\mathcal F}}\mathbb Z\cap(0,1)$ for every
    component $D_{i,a}$ appearing in the family. Equivalently, all cone angles
    $2\pi\beta_{i,a}$ have denominator dividing $N_{\mathcal F}$, and they are
    uniformly bounded away from $0$ and $2\pi$.

    \item There are constants $D,V>0$ such that
    $\operatorname{diam}(X_i,\omega_i)\leq D$ and
    $\operatorname{Vol}(X_i,\omega_i)\geq V$ for every $i$.
\end{itemize}

Using the polarized approximation lemma below, each member of $\mathcal F$ can
be approximated in the Gromov-Hausdorff sense by smooth polarized K\"ahler
manifolds with a uniform Ricci lower bound. Hence by a diagonal argument, the usual Gromov compactness and Ricci-limit-space machinery applies after approximation. We therefore pass to a subsequence and
write $(X_i,\omega_i)\longrightarrow (X_\infty,\omega_\infty)$ in the Gromov--Hausdorff sense. The limit $X_\infty$ is the main object of the
paper.

\begin{lemma}[Polarized approximation and algebraicity of the limit]
\label{lem:polarized-approximation}
For any $(X,\Delta,\omega)\in\mathcal F$, there exists a sequence of smooth
K\"ahler metrics $\omega^{(j)}$ on the smooth manifold $X$ such that, writing
\[
        L:=-N_{\mathcal F}(K_X+\Delta),
\]
we have
\begin{itemize}
    \item $L$ is a genuine holomorphic line bundle;
    \item $c_1(L)=N_{\mathcal F}[\omega^{(j)}]$;
    \item $\operatorname{Ric}(\omega^{(j)})\geq -C\omega^{(j)}$, for a constant
    $C$ independent of $j$ and of the pair in $\mathcal F$;
    \item $(X,\omega^{(j)})\to (X,\omega)$ in the Gromov--Hausdorff sense.
\end{itemize}
Consequently, after passing to a subsequence, the family
$(X_i,\omega_i)$ has a Gromov--Hausdorff limit $X_\infty$. Moreover
$X_\infty$ is homeomorphic to a normal projective variety, and the convergence
can be realized as algebraic convergence inside a fixed projective space.
\end{lemma}

\begin{proof}
The bounded denominator assumption gives
$-N_{\mathcal F}(K_X+\Delta)\in \operatorname{Pic}(X)$, so
$L:=-N_{\mathcal F}(K_X+\Delta)$ is a genuine line bundle.

The smooth approximants are obtained by regularizing the singular Hermitian
metrics and the conical potentials. In the notation of \cite[Section 2]{FGSW}, this
is the metric regularization construction: one replaces the singular
right-hand side by smooth data, solves the corresponding smooth Monge--Amp\`ere
equations, obtains uniform $L^\infty$ and diameter estimates, and recovers
Gromov--Hausdorff convergence to the original singular metric. Thus we obtain
smooth K\"ahler metrics $\omega^{(j)}\in[\omega]$ satisfying the stated Ricci
lower bound and $(X,\omega^{(j)})\longrightarrow (X,\omega)$
in the Gromov--Hausdorff sense.

The constants in this construction depend only on the dimension, the fixed
finite coefficient set, and the uniform geometric bounds defining
$\mathcal F$. In particular, the Ricci lower bound may be chosen uniformly in
both the smoothing parameter and the member of the family.

We now apply this to the sequence $(X_i,\Delta_i,\omega_i)$. For each $i$, choose
$j(i)$ sufficiently large so that $d_{GH}\big((X_i,\omega_i),(X_i,\omega_i^{(j(i))})\big)\leq i^{-1}$.
The smooth polarized manifolds
\[
        (X_i,L_i,\omega_i^{(j(i))}),
        \qquad
        L_i:=-N_{\mathcal F}(K_{X_i}+\Delta_i),
\]
have a uniform Ricci lower bound, a uniform diameter bound, and a uniform volume
lower bound. Hence Gromov compactness gives a Gromov--Hausdorff convergent
subsequence of $(X_i,\omega_i^{(j(i))})$. By the triangle inequality, the original
conical spaces $(X_i,\omega_i)$ converge to the same limit; denote it by
$X_\infty$.

Finally, \cite[Theorem~1.1]{SLI} applies to the smooth polarized
approximating sequence. Therefore, after replacing $L_i$ by a fixed power
$L_i^{k_1}$, the $X_i$ embed into a fixed projective space $\mathbb{CP}^N$, the
embedded varieties converge algebraically, and the metric limit $X_\infty$ is
homeomorphic to a normal projective variety.
\end{proof}

Since $X_\infty$ is obtained from the polarized approximation as a
non-collapsed Ricci-limit space of real dimension $2n$, we may use the
Cheeger--Colding stratification
\cite{CheegerColdingI,CheegerColdingII,CheegerColdingIII}. We write
\[
\mathcal S^k(X_\infty)
:=
{x\in X_\infty:\text{ no tangent cone at }x
\text{ splits off }\mathbb R^{k+1}}.
\]
Then
$dim_{\mathcal H}\mathcal S^k\leq k$, 
$\mathcal S(X_\infty)=\mathcal S^{2n-2}(X_\infty)$
the latter following from the codimension-two estimate for the singular set of
a non-collapsed Ricci limit.

In the K\"ahler setting the odd strata do not occur separately. Indeed, if a
tangent cone splits an $\mathbb R^{2k+1}$ factor, the parallel complex
structure on the regular part produces one further independent splitting
direction, and hence the same tangent cone splits an
$\mathbb R^{2k+2}$ factor; see
\cite[Proposition~B.5]{WangZhuBakryEmery}. Thus $\mathcal S^{2k+1}=\mathcal S^{2k}$, 
and we suppress the odd indices. In particular, $\mathcal S^0\subset\mathcal S^2\subset\cdots
\subset\mathcal S^{2n-2}=\mathcal S$.
We will use the codimension notation $S_2:=\mathcal S^{2n-2},
\qquad
S_4:=\mathcal S^{2n-4}$. 
Thus $S_2\setminus S_4$ is precisely the part of the singular set with a
codimension-two splitting defect.

\section{Local behavior}
\label{sec:local-structure}
In this section, we aim to study the metric geometry of $X_\infty$, together with the pair $(X_\infty, D_\infty)$ as a divisor in a variety. We then prove a qualitative bound on the convergence near a ``good'' tangent cone; and finally show that $(X_\infty, \omega_\infty)$ is a conical K\"ahler--Einstein singular metric. Then we degenerate one step more and get stronger structure results for any tangent cone $C(Y)$. 

We use the codimension notation $S_2:=\mathcal S^{2n-2}$ and
$S_4:=\mathcal S^{2n-4}$. Thus $S_2\setminus S_4$ denotes the points at which
some tangent cone has a real codimension-two non-Euclidean factor, but no tangent
cone has a better splitting defect.

Let $p\in S_2\setminus S_4$. By the cone-splitting theorem for Ricci-limit
spaces, together with the K\"ahler splitting structure of tangent cones
\cite{CheegerColdingII,CheegerJiangNaber,SLI}, we may choose a tangent cone of the form
$(X_\infty,r_j^{-2}d_\infty,p)\to \mathbb C_\gamma\times\mathbb C^{n-1}$ for
some sequence $r_j\to0$ and some $\gamma\in(0,1]$. Here $\mathbb C_\gamma$
denotes the flat two-dimensional cone of angle $2\pi\gamma$.

We now aim to analyze the metric on $X_\infty$ near such a point. The main issue
is to control how the divisors in the approximating pairs approach the
codimension-one singular set of the model cone. We use a tubular trapping
argument, combined with the Liu--Sz\'ekelyhidi good-coordinate theorem
\cite[Theorem~1.4 and Proposition~3.2]{SLI} and the transverse holonomy
computation in \cite[proof of Proposition~13]{CDSII}.

Pick points $p_i\in X_i$ converging to $p$ in the pointed Gromov--Hausdorff
convergence. On the rescaled spaces set $\omega_i^{(j)}:=r_j^{-2}\omega_i$.
On $X_i\setminus\operatorname{Supp}\Delta_i$ one has
$\operatorname{Ric}(\omega_i^{(j)})=\operatorname{Ric}(\omega_i)
=\lambda_i\omega_i=\lambda_i r_j^2\omega_i^{(j)}$. Thus, for fixed $j$, the
Ricci lower bound tends to zero at the rescaled scale.

Strictly speaking, the good-coordinate theorem of Liu--Sz\'ekelyhidi applies to
smooth K\"ahler metrics with Ricci curvature bounded below. We therefore apply
\cite[Theorem~1.4 and Proposition~3.2]{SLI} to the smooth polarized
approximants from Lemma~\ref{lem:polarized-approximation}, with the smoothing
parameter chosen diagonally after $r_j$ is fixed. Passing to the diagonal
sequence gives holomorphic coordinates for the conical metrics themselves. Thus,
for fixed $j$ and all sufficiently large $i$, we get domains
$\Omega_i^{(j)}\subset X_i$ and holomorphic maps
$F_i^{(j)}=(u_i^{(j)},v_{i,1}^{(j)},\dots,v_{i,n-1}^{(j)}):
\Omega_i^{(j)}\to B_1(0)\subset\mathbb C^n$, with
$\omega_i^{(j)}=i\partial\bar\partial\phi_i^{(j)}$ and
$\|\phi_i^{(j)}\|_{L^\infty}\leq C_j$ on $\Omega_i^{(j)}$.

With this notation, we prove the local volume-form estimate needed later.

\begin{lemma}[Local gauge and volume-form bounds near $S_2\setminus S_4$]
\label{lem:tubular-trapping}
Let $p\in S_2\setminus S_4$, and choose the rescaled holomorphic charts
$F_i^{(j)}:\Omega_i^{(j)}\to B\subset\mathbb C^n$ as above. Write
$D_i^{(j)}:=F_i^{(j)}(\operatorname{Supp}\Delta_i\cap\Omega_i^{(j)})$.
After shrinking $B$ to a smaller polydisc $B'\Subset B$, there are
nowhere-vanishing holomorphic functions $U_i^{(j)}$ on $B'$ such that
\[
        (\omega_i^{(j)})^n
        =
        e^{-\lambda_i r_j^2\phi_i^{(j)}}|U_i^{(j)}|^2
        \prod_a |f_{i,a}^{(j)}|^{2\beta_{i,a}-2}\,dV_{Euc} .
\]
Here the $f_{i,a}^{(j)}$ are local defining functions for the irreducible
components of $D_i^{(j)}\cap B'$, counted with coefficients
$1-\beta_{i,a}$. Moreover there is a constant $C_j$, independent of $i$, such
that $C_j^{-1}\leq |U_i^{(j)}|\leq C_j$ on $B'$.
\end{lemma}

\begin{proof}
We first trap the divisor. For every $\eta>0$, after first taking $j$ sufficiently
large and then $i$ sufficiently large, we claim that
$D_i^{(j)}\cap \frac12B\subset N_\eta(\{u=0\})$. If not, there are
$\eta_0>0$ and points $x_i\in D_i^{(j)}\cap \frac12B$ whose distance from
$\{u=0\}$ is at least $\eta_0$. Passing to a subsequence, $x_i$ converges to a
point in the regular part of the model cone
$\mathbb C_\gamma\times\mathbb C^{n-1}$ away from $\{u=0\}$. In this region the
rescaled metrics converge smoothly to the flat model, so the volume density at
$x_i$ tends to $1$. On the other hand, every point of
$\operatorname{Supp}\Delta_i$ has tangent density bounded away from $1$ by a
constant depending only on $N_{\mathcal F}$; at a smooth point of a single
component the density is $\beta_{i,a}\leq 1-\frac1{N_{\mathcal F}}$, and at
intersections or singular points the density is no larger. This contradiction
proves the trapping claim.

Choose $0<r_u,r_v<1/2$ so that
$\Delta_{r_u}\times B_{r_v}^{n-1}\Subset \frac12 B_1(0)$, and replace the
chart target by the product polydisc
$B:=\Delta_{r_u}\times B_{r_v}^{n-1}$. The trapping implies that, for $j$ and then
$i$ sufficiently large, $D_i^{(j)}$ does not meet the side region
$\{|u|\geq r_u,\ |v|<r_v\}$. Hence the projection $\pi(u,v)=v$ restricts to a
proper holomorphic map $D_i^{(j)}\cap B\to B_{r_v}^{n-1}$. By Weierstrass
preparation, after shrinking once more if necessary,
$D_i^{(j)}\cap B=\{P_i^{(j)}(u,v)=0\}$, where
\[
        P_i^{(j)}(u,v)
        =
        u^{m_i^{(j)}}
        +a_{i,m_i^{(j)}-1}^{(j)}(v)u^{m_i^{(j)}-1}
        +\cdots+a_{i,0}^{(j)}(v).
\]

We next bound $m_i^{(j)}$. Pick $v_0\in B_{r_v}^{n-1}$ so that the vertical disk
$T_i^{(j)}:=\{(u,v_0): |u|<r_u\}$ is transverse to $D_i^{(j)}$ for all $i$ in
the chosen subsequence. For fixed $v_0$, the polynomial
$P_i^{(j)}(u,v_0)$ may be factored as a one-variable polynomial, with roots
counted with multiplicity. 

From the Chern connection on $T^{1,0}$ we obtain the induced connection on the
anticanonical bundle, whose monodromy around a loop lies in $U(1)$. In the
following take $T:=T_i^{(j)}$ and
$T_\epsilon:=T\setminus\bigcup_a B_\epsilon(p_a)$, where
$p_a\in T\cap D_i^{(j)}$.
Then $Hol(\partial T) = (\prod_a Hol (\partial B_\epsilon(p_a))) \exp(\int_{T_\epsilon} \rho)$. Now, as $\epsilon \rightarrow 0$, each small loop around a transverse conic point contributes the local monodromy $\exp(2 \pi i (1 - \beta_i))$. Applying the holonomy computation in the proof of
\cite[Proposition~13]{CDSII} to the punctured disk
$T\setminus D_i^{(j)}$, one obtains
\[
        \operatorname{Hol}(\partial T)
        =
        \exp\left(
        2\pi\sqrt{-1}
        \left[
        \sum_a m_a(1-\beta_{i,a})
        +\int_T\rho_i^{(j)}
        \right]\right),
\]
where $\rho_i^{(j)}$ is the Ricci form of the rescaled metric on the smooth
locus. The equality is first an equality in $S^1$; we choose the branch
determined by the limiting transverse cone $\mathbb C_\gamma$.

The Chern--Levine--Nirenberg estimate used in the proof of
\cite[Proposition~13]{CDSII} applies here because the Liu--Sz\'ekelyhidi chart gives
$\|\phi_i^{(j)}\|_{L^\infty}\leq C_j$. Hence
$\operatorname{Area}_{\omega_i^{(j)}}(T_i^{(j)})\leq C_j$. Since
$\rho_i^{(j)}=\lambda_i r_j^2\omega_i^{(j)}$ on the smooth locus, we get
$|\int_{T_i^{(j)}}\rho_i^{(j)}|\leq C_j r_j^2$. Letting $i\to\infty$ and then
$j\to\infty$, the holonomy converges to the holonomy of the transverse cone,
namely $\exp(2\pi\sqrt{-1}(1-\gamma))$. Thus
$\sum_a m_a(1-\beta_{i,a})\to 1-\gamma$. Since
$1-\beta_{i,a}\in \frac1{N_{\mathcal F}}\mathbb Z$, the degrees
$m_i^{(j)}$ are uniformly bounded, and $\gamma\in\frac1{N_{\mathcal F}}\mathbb Z$.

The trapping also gives uniform control of the Weierstrass coefficients. For
each fixed $j$, after passing to a subsequence in $i$, the coefficients
$a_{i,\ell}^{(j)}$ converge uniformly on compact subsets to holomorphic limits
$a_{\infty,\ell}^{(j)}$. Their roots remain in $|u|\leq\eta_j$; hence, after
subsequently taking $j\to\infty$ diagonally, the limiting Weierstrass divisor is
supported on $\{u=0\}$.

Let $q_{i,a}:=N_{\mathcal F}(1-\beta_{i,a})$. Recall that
$\Delta_i=\sum_a(1-\beta_{i,a})D_{i,a}$, so that
$N_{\mathcal F}\Delta_i=\sum_a q_{i,a}D_{i,a}$ is integral. On the polydisc
$B$, choose the local defining functions $f_{i,a}^{(j)}$ coming from the
Weierstrass preparation above, normalized to be monic in the $u$-variable. If
$\Omega:=du\wedge dv_1\wedge\cdots\wedge dv_{n-1}$, then
$\tau_i^{(j)}
:=
\Omega^{\otimes N_{\mathcal F}}
\prod_a(f_{i,a}^{(j)})^{-q_{i,a}}
$
is a meromorphic $N_{\mathcal F}$-canonical form whose pole divisor is
$N_{\mathcal F}\Delta_i$. Equivalently, $\tau_i^{(j)}$ is a nowhere-vanishing
holomorphic frame of $K_B^{\otimes N_{\mathcal F}}
\otimes \mathcal O_B(N_{\mathcal F}\Delta_i)
=
\mathcal O_B\bigl(N_{\mathcal F}(K_B+\Delta_i)\bigr).
$

On $B\setminus D_i^{(j)}$, the conical K\"ahler--Einstein equation implies that
\[
H_i^{(j)}
:=
\log\frac{(\omega_i^{(j)})^n}{dV}
+\lambda_i r_j^2\phi_i^{(j)}
-\sum_a(\beta_{i,a}-1)\log|f_{i,a}^{(j)}|^2
\]
is pluriharmonic. The logarithmic terms have exactly the singularities
prescribed by $\Delta_i$, so $H_i^{(j)}$ is locally bounded across
$D_i^{(j)}$ and therefore extends pluriharmonically across the divisor. After
shrinking $B$ once, we may assume that it is simply connected, and hence write
$H_i^{(j)}=\log|U_i^{(j)}|^2$ for a nowhere-vanishing holomorphic function
$U_i^{(j)}$. Thus
\[
(\omega_i^{(j)})^n
=
e^{-\lambda_i r_j^2\phi_i^{(j)}}
|U_i^{(j)}|^2
\prod_a |f_{i,a}^{(j)}|^{2\beta_{i,a}-2}\,dV.
\]

It remains to prove the two-sided bound for $U_i^{(j)}$. Fix, for the rest of
the proof, polydiscs $B'\Subset B_1\Subset B$. Here $j$ is fixed, and the
shrinking is chosen independently of $i$; all constants below are allowed to
depend on $j$. Since the $f_{i,a}^{(j)}$ are monic Weierstrass polynomials,
their roots remain in a fixed bounded $u$-disc by the trapping argument, while
their degrees are uniformly bounded by the transverse holonomy estimate.
Therefore $\sup_{B_1}|f_{i,a}^{(j)}|\leq C_j$
uniformly in $i$, and there are only uniformly many factors. Since
$2\beta_{i,a}-2\leq0$, it follows that $\prod_a |f_{i,a}^{(j)}|^{2\beta_{i,a}-2}\geq C_j^{-1}$
on $B_1\setminus D_i^{(j)}$.

We also have a uniform local volume bound on $B_1$. Indeed, by the construction
of the Liu--Sz\'ekelyhidi chart, for fixed $j$ the set
$(F_i^{(j)})^{-1}(B_1)$ is contained in a metric ball
$B_{\omega_i^{(j)}}(p_i,R_j)$, where $R_j$ is independent of $i$. The
rescaled metrics have a uniform Ricci lower bound for fixed $j$, and hence
Bishop--Gromov comparison gives $\int_{B_1}(\omega_i^{(j)})^n\leq C_j$
Using the volume-form equation, the bound
$\|\phi_i^{(j)}\|_{L^\infty}\leq C_j$, and the lower bound for the singular
weight, we obtain $\int_{B_1}|U_i^{(j)}|^2\,dV\leq C_j$
The mean-value inequality for holomorphic functions then gives
$|U_i^{(j)}|\leq C_j$ on $B'\Subset B_1$.

For the lower bound, suppose no such lower bound holds. Then, after passing to a
subsequence, there are points $x_i\in B'$ with $U_i^{(j)}(x_i)\to0$. By the
upper bound and Montel's theorem, $U_i^{(j)}$ converges uniformly on compact
subsets to a holomorphic function $U_\infty^{(j)}$. Since each $U_i^{(j)}$ is
nowhere vanishing, Hurwitz's theorem implies that the limit is either nowhere
vanishing or identically zero. Since it vanishes at the limit of the $x_i$, it
must be identically zero.

Fix $B''\Subset B'$. We use the elementary one-variable estimate that if
$\alpha_a\in[0,1)$ and $\sum_a\alpha_a\leq\delta_*<1$, then
$\prod_a |u-\sigma_a|^{-2\alpha_a}$ is uniformly $L^1$ on compact subdisks,
uniformly in the roots $\sigma_a$. Applying this estimate fiberwise in the
$u$-direction, and using the holonomy bound
$\sum_a(1-\beta_{i,a})\leq\delta_*<1$ for $j$ fixed sufficiently large, the
weights $\prod_a |f_{i,a}^{(j)}|^{2\beta_{i,a}-2}$ are uniformly integrable on
$B''$. Since $U_i^{(j)}\to0$ uniformly on $B''$ and the potentials
$\phi_i^{(j)}$ are uniformly bounded, the right-hand side of the volume-form
equation converges to $0$ in $L^1(B'')$. Hence
$\int_{B''}(\omega_i^{(j)})^n\to0$.

This contradicts non-collapsing and volume convergence in the Liu--Sz\'ekelyhidi chart, since
for fixed $j$ the chart is Gromov--Hausdorff close to a definite ball in
$\mathbb C_\gamma\times\mathbb C^{n-1}$. Therefore
$C_j^{-1}\leq |U_i^{(j)}|\leq C_j$ on $B'$, after shrinking $B'$ if necessary.
\end{proof}

We now combine the tubular trapping lemma with Bedford--Taylor continuity. The
point is that the previous lemma gives a genuine local logarithmic gauge for the
approximating volume forms; after passing to the limit, the only possible
codimension-one singular density is the model density $|u|^{2\gamma-2}$.

\begin{lemma}[Local structure away from codimension four]
\label{lem:local-structure}
For every $p\in X_\infty^{reg}\cup(S_2\setminus S_4)$ there is a
neighborhood $U$ of $p$ and a holomorphic embedding
$F:U\to B\subset\mathbb C^n$ with coordinates $(u,v_1,\dots,v_{n-1})$ on $B$
such that exactly one of the following holds.

\begin{enumerate}
\item If $p\in X_\infty^{reg}$, then $F(U)$ is smooth and $\omega_\infty$ is a
smooth K\"ahler metric on $U$.

\item If $p\in S_2\setminus S_4$, then, after shrinking $U$, the codimension-one
metric singular set in $U$ is $D=\{u=0\}$, and on $U\setminus D$ one has
$\omega_\infty=i\partial\bar\partial\phi$ and
\[
        \omega_\infty^n
        =
        \exp(-\lambda\phi+G)|u|^{2\gamma-2}dV .
\]
Here $\phi$ is bounded plurisubharmonic, $G$ is bounded pluriharmonic, and
$\gamma\in \frac1{N_{\mathcal F}}\mathbb Z\cap(0,1)$. With respect to these
coordinates, $\omega_\infty$ is uniformly equivalent on a smaller ball to the
standard cone metric
\[
        \omega_\gamma
        :=
        \gamma^2 |u|^{2\gamma-2}i\,du\wedge d\bar u
        +
        \sum_{a=1}^{n-1}i\,dv_a\wedge d\bar v_a .
\]
In particular every point of $D$ lies in $S_2\setminus S_4$.
\end{enumerate}

On overlaps, if $\phi_\alpha$ and $\phi_\beta$ are local bounded potentials for
$\omega_\infty$, then $i\partial\bar\partial(\phi_\alpha-\phi_\beta)=0$; hence
$\phi_\alpha-\phi_\beta$ is pluriharmonic.
\end{lemma}

\begin{proof}
The dimension estimate for $\Sigma$ follows from the Cheeger--Colding
stratification: $\dim_{\mathcal H}\mathcal S^{2n-4}\leq 2n-4$, and in the
K\"ahler setting the odd real strata do not contribute to the codimension-one
analysis.

If $p\in X_\infty^{reg}$, choose scales $r_j\to0$ such that
$(X_\infty,r_j^{-2}d_\infty,p)\to \mathbb C^n$. By the smooth convergence on
the regular set, equivalently by the Liu--Sz\'ekelyhidi good-coordinate theorem
applied through the polarized approximation, we obtain a holomorphic chart in
which $\omega_\infty$ is a smooth K\"ahler metric. This proves the first case.

Now assume $p\in S_2\setminus S_4$. Choose a tangent cone
$(X_\infty,r_j^{-2}d_\infty,p)\to \mathbb C_\gamma\times\mathbb C^{n-1}$.
Apply Lemma~\ref{lem:tubular-trapping} in the corresponding rescaled Liu--Sz\'ekelyhidi
coordinates. Thus, on a smaller polydisc $B'$, the approximating metrics satisfy
\[
        (\omega_i^{(j)})^n
        =
        e^{-\lambda_i r_j^2\phi_i^{(j)}}|U_i^{(j)}|^2
        \prod_a |f_{i,a}^{(j)}|^{2\beta_{i,a}-2}\,dV ,
\]
with $C_j^{-1}\leq |U_i^{(j)}|\leq C_j$, and the divisors are trapped in a
tubular neighborhood of $\{u=0\}$. The same lemma also gives
$\gamma\in \frac1{N_{\mathcal F}}\mathbb Z\cap(0,1)$.

After passing to a subsequence, the uniformly bounded plurisubharmonic
potentials $\phi_i^{(j)}$ converge locally in $L^1$ to a bounded
plurisubharmonic function $\phi^{(j)}$. On every compact subset of
$B'\setminus\{u=0\}$, the trapping lemma excludes the divisors for $i\gg1$,
and the convergence on the regular set is smooth; hence the limiting
Monge--Amp\`ere equation holds there. Since $\phi^{(j)}$ is locally bounded,
its Bedford--Taylor Monge--Amp\`ere measure does not charge the divisor
$\{u=0\}$ \cite{BedfordTaylorCapacity}. The density
$|u|^{2\gamma-2}dV$ is locally integrable and also gives this divisor zero
mass. Therefore the same equation holds on all of $B'$:
$(i\partial\bar\partial\phi^{(j)})^n
=e^{-\lambda_\infty r_j^2\phi^{(j)}}|U_\infty^{(j)}|^2
|u|^{2\gamma-2}dV$.
Since $U_\infty^{(j)}$ is holomorphic and nowhere vanishing,
$G^{(j)}:=\log |U_\infty^{(j)}|^2$ is bounded and pluriharmonic.

Passing from the rescaled chart back to the original scale, and then letting
$j\to\infty$ along the chosen tangent-cone sequence, gives a bounded
plurisubharmonic potential $\phi$ and a bounded pluriharmonic function $G$ such
that on $U\setminus\{u=0\}$,
$\omega_\infty=i\partial\bar\partial\phi$ and
$\omega_\infty^n=\exp(-\lambda\phi+G)|u|^{2\gamma-2}dV$.

It remains only to identify the local metric model. Note that the preceding
equation is now in the form used in \cite[Section~3.2]{CDSII}, in the smooth
divisor chart $D=\{u=0\}$. Proposition~24 of \cite{CDSII} gives uniform
equivalence with the standard cone metric $\omega_\gamma$ on a smaller ball,
and \cite[Proposition~27]{CDSII} gives the standard cone regularity of the
potential.
Consequently the metric singular set in this chart is exactly $\{u=0\}$, and
each point of this divisor has tangent cone
$\mathbb C_\gamma\times\mathbb C^{n-1}$. Hence $D\subset S_2\setminus S_4$.

Finally, if $\phi_\alpha$ and $\phi_\beta$ are two local bounded potentials for
the same current $\omega_\infty$, then
$i\partial\bar\partial(\phi_\alpha-\phi_\beta)=0$ on the overlap. Since the
difference is locally bounded, it is pluriharmonic. This proves the lemma.
\end{proof}

We now prove the main theorem of this section.

\begin{theorem}
The Gromov--Hausdorff limit $X_\infty$ is 
\begin{itemize}
    \item Normal, projective, and log-$\mathbb{Q}$-Gorenstein
    \item There exists an effective $\mathbb{Q}$-divisor $\Delta = \sum_i (1-\beta_i) D_i$ on $X_\infty$, where the $D_i$ are the irreducible codimension-one components of the divisorial conical locus, such that $K_{X_\infty} + \Delta$ is $\mathbb{Q}$-Cartier
    \item $\omega_\infty$ satisfies the weak conical KE equation. More precisely, there exists a closed positive $(1,1)$ current $\omega_0\in c_1\bigl(-(K_{X_\infty}+\Delta)\bigr)$, a function $\phi \in PSH(X_\infty, \omega_0) \cap L^\infty$, Hermitian metrics $h_i$ on $\mathcal O(D_i)$, and a bounded function $F \in L^\infty(X_\infty)$ such that
$\omega_\infty = \omega_0 + i \partial \bar{\partial} \phi$
and
$(\omega_0 + i \partial \bar{\partial} \phi)^n = \omega_0^n \exp(-\lambda \phi + F) \prod_i |s_i|_{h_i}^{2(\beta_i-1)}$
in the Bedford--Taylor sense. On $X_\infty^{reg} \backslash \mathrm{Supp}(\Delta)$, the function $F$ satisfies
$i \partial \bar{\partial} F = \operatorname{Ric}(\omega_0) - \lambda \omega_0 - \sum_i (1-\beta_i)\Theta(h_i).$
\end{itemize} 
\end{theorem}

\begin{proof}
By the polarized approximation lemma and the uniform estimates in \cite{SLI}, Theorem 1.1, we obtain embeddings of the approximants into a fixed projective space $\mathbb{CP}^N$. By a diagonal argument, the limit $X_\infty$ is also realized as a projective variety. The same theorem implies that $X_\infty$ is normal.

Let $\mathcal R$ and $\mathcal S$ denote the metric regular and singular sets.
Let $\varepsilon_{\mathrm{LS}}>0$ be the constant in the
Liu--Sz\'ekelyhidi $\varepsilon$-regularity theorem. Since the cone
coefficients belong to a fixed finite subset of $(0,1)$, there is
$\delta_{\mathrm{ang}}>0$ such that every codimension-two conical point has
volume density at most $1-\delta_{\mathrm{ang}}$. Choose $0<\varepsilon_0<
 \min\{\varepsilon_{\mathrm{LS}},\delta_{\mathrm{ang}}/2\}$.
If a point belongs to the quantitative regular set $\mathcal R_{\varepsilon_0}$,
then the $\varepsilon$-regularity theorem gives a holomorphic regular chart in
a neighborhood of that point. Conversely, at every metrically regular point
the volume ratios tend to one. Hence $\mathcal R=\mathcal R_{\varepsilon_0}$.
In particular, $\mathcal R$ is open and $\mathcal S$ is closed.

We next use the quantitative form of the Siu argument in
\cite[Section~4]{SLI}. At an algebraically smooth point, failure of metric
regularity is detected by a positive Lelong number of the limiting Ricci
current. The fixed density gap above gives a uniform constant $c_0>0$ such
that $\mathcal S
=
\operatorname{Sing}_{\mathrm{an}}(X_\infty)
\cup
\left\{
 x\in X_\infty^{\mathrm{an,reg}}:
 \nu\bigl(\operatorname{Ric}(\omega_\infty),x\bigr)\ge c_0
\right\}.
$
The second set is analytic by Siu's theorem, and the first is analytic because
$X_\infty$ is a normal variety. Thus $\mathcal S$ is a complex analytic
subvariety; denote it by $V_{\mathcal S}$. Let $D_1,\ldots,D_m$ be the
irreducible codimension-one components of $V_{\mathcal S}$.

We now assign the cone coefficient along each $D_i$. The point is that the coefficient is first defined only on the good locus of $D_i$, and then extended by irreducibility. Set
$D_i^\circ:=D_i^{reg}\cap X_\infty^{reg}\cap(S_2\setminus S_4)$.
This is the locus where $D_i$ is a smooth hypersurface in the smooth part of $X_\infty$ and where the metric tangent cone has exactly a codimension-two non-Euclidean factor. The complement of $D_i^\circ$ in $D_i$ is small in the following sense. Since $X_\infty$ is normal, $X_\infty^{sing}$ has complex codimension at least two in $X_\infty$, and hence cannot contain an irreducible divisor. The singular locus of $D_i$ is a proper analytic subset of $D_i$. Finally, $S_4$ has Hausdorff dimension at most $2n-4$, whereas $D_i$ has real dimension $2n-2$. Hence $S_4$ cannot contain a relatively open subset of $D_i^{reg}$. Consequently $D_i^\circ$ is dense in $D_i$.

Let $p\in D_i^\circ$. By the local structure theorem, after shrinking around $p$ there are holomorphic coordinates $(u,v_1,\dots,v_{n-1})$ on $X_\infty^{reg}$ such that $D_i\cap U=\{u=0\}$ and
$\omega_\infty=i\partial\bar\partial\phi_\alpha,\qquad \omega_\infty^n=\exp(-\lambda\phi_\alpha+G_\alpha)|u|^{2\beta(p)-2}dV$
on $U\setminus D_i$, where $\phi_\alpha$ is bounded plurisubharmonic, $G_\alpha$ is bounded pluriharmonic, and $\beta(p)\in \frac{1}{N_{\mathcal F}}\mathbb Z\cap(0,1)$. The exponent is intrinsic: on the overlap of two such charts, the local defining functions of $D_i$ differ by a nowhere-vanishing holomorphic unit, so the order of the singular density along $D_i$ is unchanged. Hence $\beta(p)$ is locally constant on $D_i^\circ$.

Since $D_i$ is irreducible and $D_i^\circ$ is dense, this locally constant value is independent of the chosen general point of $D_i$. Equivalently, any two general points of $D_i^\circ$ can be joined through the smooth divisorial locus while avoiding the codimension-two exceptional set, and the overlap argument above forces the same exponent along the chain of local structure charts. We denote this common value by $\beta_i$ and define
$\Delta:=\sum_{i=1}^m(1-\beta_i)D_i$.

We now show that $K_{X_\infty}+\Delta$ is $\mathbb Q$-Cartier. Let
$\mathcal L_\infty:=\mathcal O_{\mathbb P^N}(1)|_{X_\infty}$. By
\cite[Theorem~1.1]{SLI}, the projective embeddings are defined by a fixed power
$L_i^{k_1}$, so $\mathcal L_\infty$ is the limiting polarization for
$L_i^{k_1}$. On the log-smooth locus, convergence of the local polarized frames
identifies $\mathcal L_\infty$ with
$\mathcal O_{X_\infty}(-k_1N_{\mathcal F}(K_{X_\infty}+\Delta))$ as rank-one
reflexive sheaves. Since the left-hand side is locally free, the right-hand side
is locally free as well. Thus $K_{X_\infty}+\Delta$ is $\mathbb Q$-Cartier.

Choose a Hermitian metric $h_{\mathcal L_\infty}$ on $\mathcal L_\infty$ and set
$\omega_0:=(k_1N_{\mathcal F})^{-1}\Theta(h_{\mathcal L_\infty})$. Then
$\omega_0$ represents $c_1(-(K_{X_\infty}+\Delta))$. Locally, write
$\omega_0=i\partial\bar\partial\psi_\alpha$.
Since $\omega_\infty$ is also locally given by bounded plurisubharmonic potentials, the differences $\phi_\alpha - \psi_\alpha$ are pluriharmonic on overlaps. Hence they glue to give a global bounded function $\phi \in PSH(X_\infty,\omega_0) \cap L^\infty$ such that $\omega_\infty = \omega_0 + i \partial \bar{\partial} \phi$.

Now choose Hermitian metrics $h_i$ on $\mathcal O(D_i)$, and let $s_i$ denote the defining section of $D_i$. On a local chart $U_\alpha$, write $\omega_0^n = e^{H_\alpha} dV_\alpha$, and write $|s_i|_{h_i}^2 = |f_{i,\alpha}|^2 e^{-k_{i,\alpha}}$, where $f_{i,\alpha}$ is a local defining function for $D_i$. We then define
$F_\alpha := G_\alpha - \lambda \psi_\alpha - H_\alpha - \sum_i (1-\beta_i)k_{i,\alpha}.$
Substituting into the local equation for $\omega_\infty^n$, we obtain $\omega_\infty^n = \omega_0^n \exp(-\lambda \phi + F_\alpha)\prod_i |s_i|_{h_i}^{2(\beta_i-1)}$ on $U_\alpha \cap X_\infty^{reg}$.

The functions $F_\alpha$ agree on overlaps, so they glue to a global bounded function $F \in L^\infty(X_\infty^{reg} \cup (S_2 \backslash S_4))$. This gives
$(\omega_0 + i \partial \bar{\partial}\phi)^n = \omega_0^n \exp(-\lambda \phi + F)\prod_i |s_i|_{h_i}^{2(\beta_i-1)}$
globally in the Bedford--Taylor sense. 

Finally, on $X_\infty^{reg} \backslash \mathrm{Supp}(\Delta)$, the local Einstein equation implies $i \partial \bar{\partial} G_\alpha = 0$, and therefore$i \partial \bar{\partial} F = \operatorname{Ric}(\omega_0) - \lambda \omega_0 - \sum_i (1-\beta_i)\Theta(h_i).$
This is exactly the weak conical K\"ahler--Einstein equation for the pair $(X_\infty,\Delta)$.
\end{proof}

Now fix $p\in X_\infty$, and let $r_i\to0$ be a sequence such that
\[
(X_\infty,r_i^{-2}\omega_\infty,p)
\longrightarrow
(C(Y),\omega_{C(Y)},o)
\]
in the pointed Gromov--Hausdorff sense. We get a similar structural result for $C(Y)$. 

\begin{theorem}
\label{thm:tangent-cone-log-structure}
The metric tangent cone $C(Y)$ is:
\begin{itemize}
    \item Normal, affine, and log-$\mathbb{Q}$-Gorenstein
    \item There exists an effective $\mathbb{Q}$-divisor $\Delta = \sum_i (1-\beta_i) D_i$ on $C(Y)$, where the $D_i$ are the irreducible codimension-one components of the divisorial conical locus, such that $K_{C(Y)} + \Delta$ is $\mathbb{Q}$-Cartier 
    \item $\omega_{C(Y)}$ satisfies the weak conical Ricci-flat equation. More precisely, there is a section $s$ of $K_{C(Y)}^N$, metric $h_i$, sections $s_i$ on $\mathcal O(D_i)$ such that $\omega_{C(Y)}^n = e^F(s \wedge \bar{s})^{1/N} \prod |s_i|_{h_i}^{2(\beta_i - 1)}$, and the function $F$ satisfies $i \partial \overline{\partial} F = \sum_i (1- \beta_i) \Theta(h_i)$.
\end{itemize}
\end{theorem}

\begin{proof}
   By \cite[Section~4]{SLII}, $C(Y)$ is an affine algebraic variety. The same quantitative $\varepsilon$-regularity and Siu argument applies on the
  tangent cone, with the same uniform angle gap. Hence its metric singular locus
is a closed analytic subvariety $V_{\mathcal S}(C(Y))$. Let
$D_1,\ldots,D_m$ be its irreducible codimension-one components. 

   Let $Z = C(Y)^{an, sing} \cup \Delta^{sing} \cup (S \backslash Supp(\Delta))$. By construction, $Z$ is an analytic subset of complex codimension at least two. On $C(Y) \backslash A$, every point is either metrically regular; or lies on a unique smooth component of $\Delta$ is an element in $S_2 \backslash S_4$. We can cover it by three type of charts 

    \begin{itemize}
    \item $U_\alpha \cap \Delta = \empty$, then $\omega^n_{C(Y)}  = e^{G_a} dV_a$
    \item $U_\alpha \cap \Delta = \{u_\alpha = 0 \}$, then $\omega_{C(Y)}^n = e^{G_\alpha}
|u_\alpha|^{2(\beta_\alpha-1)}\,dV_\alpha$
    \end{itemize}

    On compact subsets of $C(Y)^{reg}\setminus\operatorname{Supp}(\Delta)$, the
rescaled Einstein metrics converge smoothly by Anderson's regular-set
compactness theorem \cite{AndersonRicci}; since
$\operatorname{Ric}(r_i^{-2}\omega_\infty)=\lambda r_i^2(r_i^{-2}\omega_\infty)$,
the limit metric is Ricci-flat there. Hence $dd^cG_a=0$ on both types of charts. Thus, possibly after shrinking $U_a$ if needed, pick a holomorphic function $g_a$ such that $R(g_a) = G_a$. 

    Now, on a chart of the first type, define $\tau_\alpha := e^{-N g_\alpha} (dz_{\alpha,1} \wedge \cdots \wedge dz_{\alpha,n})^{\otimes N}$. And on a chart of the second type, define $\tau_\alpha := e^{-N g_\alpha} (du_\alpha \wedge dv_{\alpha,1} \wedge \cdots \wedge dv_{\alpha,n-1})^{\otimes N} u_\alpha^{-N(1-\beta_\alpha)}$. Here $\tau_\alpha$ is viewed as a holomorphic section of $\omega_{U_\alpha}^{\otimes N}(N\Delta|_{U_\alpha})$, and this is legitimate because $N(1-\beta_\alpha) \in \mathbb{Z}_{> 0}$.
    
    By construction, $(\tau_\alpha \otimes \bar{\tau}_\alpha)^{1/N} = \omega_{C(Y)}^n$ on $U_\alpha \setminus \operatorname{Supp}(\Delta)$. Therefore, on an overlap $U_\alpha \cap U_\beta \subset C(Y) \setminus A$, the ratio $h_{\alpha\beta} := \tau_\alpha / \tau_\beta$ is a nowhere-vanishing holomorphic function satisfying $|h_{\alpha\beta}| = 1$. Hence $h_{\alpha\beta}$ is constant. Thus the $\tau_\alpha$ glue to a nowhere-vanishing holomorphic section $\tau \in H^0(C(Y)\setminus A, \omega^{\otimes N}(N\Delta))$.

    Since $A$ has complex codimension at least $2$ and $C(Y)$ is normal, the reflexive sheaf $\omega_{C(Y)}^{[N]}(N\Delta)$ is the unique extension of $\omega^{\otimes N}(N\Delta)$ from $C(Y)\setminus A$. Therefore $\tau$ extends uniquely across $A$ to a section $s_\Delta \in H^0(C(Y), \omega_{C(Y)}^{[N]}(N\Delta))$. And the section extends as well by Shiffman's theorem. 
    
    This extension is nowhere vanishing: if it vanished somewhere, its zero locus would contain a codimension-one component, hence would meet $C(Y)\setminus A$, contradicting the fact that $\tau$ is nowhere vanishing there. It follows that $\omega_{C(Y)}^{[N]}(N\Delta)$ is locally free, generated locally by $s_\Delta$. Thus $N(K_{C(Y)}+\Delta)$ is Cartier, and $K_{C(Y)}+\Delta$ is $\mathbb{Q}$-Cartier.
    
    Finally, on $C(Y)^{reg}\setminus \operatorname{Supp}(\Delta)$, the identity $(s_\Delta \otimes \bar{s}_\Delta)^{1/N} = \omega_{C(Y)}^n$ holds by construction. If $f_\alpha$ is a local defining function for $E_j$ and $h_j$ is a Hermitian metric on $O(E_j)$, then $i\partial \bar{\partial} \log |f_\alpha|_{h_j}^2 = [E_j] - \Theta(h_j)$, so the local equation above is equivalent to $i\partial \bar{\partial} G_\alpha = \sum_j (1-\beta_j)\Theta(h_j)$ away from the divisor, as desired. 
    The same identity also shows that $(C(Y),\Delta)$ is klt: the local
log-canonical measure determined by $s_\Delta$ has finite mass because it agrees
with the metric volume measure. We will use this integrability in the uniqueness
argument below.
    
\end{proof}

The following lemma records the precise form of boundary convergence that will be
used later in the degeneration argument. The point is that the codimension-one
boundary is not tracked by Hausdorff-measure semicontinuity. Instead, we use the
local holomorphic charts from the structure theorem and the tubular trapping
lemma: in those charts the boundary is given by Weierstrass equations of
uniformly bounded degree, so Bishop compactness gives analytic limits of the
supports.

\begin{lemma}[Convergence of the codimension-one boundary under blow-up]
\label{lem:boundary-support-convergence}
Let $r_i\to0$ be such that
$(X_\infty,r_i^{-2}d_\infty,p)\to (C(Y),d_C,o)$ in the pointed
Gromov--Hausdorff sense. For $Z\in\{X_\infty,C(Y)\}$, write
$S_2(Z):=\mathcal S^{2n-2}(Z)$ and $S_4(Z):=\mathcal S^{2n-4}(Z)$, and set
$S'(Z):=S_2(Z)\setminus S_4(Z)$. Define
$\Delta_Z:=\overline{S'(Z)}$, where the closure is taken in the complex analytic
structure given by the structure theorem.

Then $\Delta_Z^{reg}\cap S'(Z)$ is dense in $\Delta_Z$, and
$\Delta_Z\setminus(\Delta_Z^{reg}\cap S'(Z))$ has locally finite
$H^{2n-4}$-measure.

Moreover, the supports $\Delta_{X_\infty}$, viewed inside the rescaled spaces
$(X_\infty,r_i^{-2}d_\infty,p)$, converge locally in the Kuratowski sense to
$\Delta_{C(Y)}$. Equivalently, after choosing any pointed
Gromov--Hausdorff realization of the convergence, the following two statements
hold.

\begin{enumerate}
\item If $x_i\in\Delta_{X_\infty}$ and $x_i$, viewed in
$(X_\infty,r_i^{-2}d_\infty,p)$, converges to $q\in C(Y)$, then
$q\in\Delta_{C(Y)}$.

\item If $q\in\Delta_{C(Y)}$, then there exist points
$x_i\in\Delta_{X_\infty}$ whose images in the rescaled spaces converge to $q$.
\end{enumerate}
\end{lemma}

\begin{proof}
We first prove the inclusions $(\Delta_Z^{smooth} \cap S'_Z) \subset \Delta_Z^{smooth} \subset  \overline{S_Z'} = \Delta_Z$. That $(\Delta_Z^{smooth} \cap S'_Z) \subset \Delta_Z^{smooth}$ and $S_Z \subset \overline{S'_Z}$ is tautological; First, we show that $\Delta_Z^{smooth} \cap S'Z$ is dense in $\Delta_Z^{smooth}$. Let $\Sigma_Z := Z \setminus (Z^{reg} \cup S_2(Z))$. By the structure theorem on $X\infty$, and by its cone-side analogue on $C(Y)$, we have $\dim_H(\Sigma_Z) \leq 2n-4$. Suppose that $\Delta_Z^{smooth} \cap S'_Z$ is not dense in $\Delta_Z^{smooth}$. Then there is a nonempty relatively open set $U \subset \Delta_Z^{smooth}$ such that $U \cap S'_Z = \emptyset$. Since $\Sigma_Z$ has Hausdorff codimension at least $2$, it has empty interior in the smooth hypersurface $\Delta_Z^{smooth}$, so after shrinking $U$ we may assume $U \cap \Sigma_Z = \emptyset$. Now $U \subset \Delta_Z \subset S$, hence every point of $U$ is singular, and since $U \cap \Sigma_Z = \emptyset$, every point of $U$ lies in $S_2(Z)$. But $U \cap S'_Z = \emptyset$, so in fact $U \subset S_4(Z)$. This is impossible, since $U$ has real dimension $2n-2$, whereas $\dim_H S_4(Z) \leq 2n-4$. Thus $\Delta_Z^{smooth} \cap S'_Z$ is dense in $\Delta_Z^{smooth}$.

Now we show that $S'_Z \subset \Delta_Z$. Indeed, this follows from the structure theorem: If $p \in S'_Z$, then there is a neighborhood $U$ of $p$ such that $U \cap S \subset \Delta_Z$; thus $\overline{S'_Z} \subset \overline{\Delta_Z} = \Delta_Z$. Which proves both of the inclusions. $H^{2n-4}(\Delta_Z \backslash (\Delta_Z^{smooth} \cap S'_Z)) < \infty $ follows easily from the proofs.

We first prove (1). By density of $\Delta_{X_\infty}^{smooth}\cap S'(X_\infty)$
in $\Delta_{X_\infty}$, it is enough to consider sequences
$x_i\in \Delta_{X_\infty}^{smooth}\cap S'(X_\infty)$, after changing $x_i$ by
$o(r_i)$ if necessary. Suppose that $x_i$, viewed in
$(X_\infty,r_i^{-2}d_\infty,p)$, converges to $q\in C(Y)$, and assume
$q\notin\Delta_{C(Y)}$. Since $\Delta_{C(Y)}$ is closed, choose
$\delta>0$ such that $B_{2\delta}(q)\cap\Delta_{C(Y)}=\emptyset$.

Then every singular point in $B_{2\delta}(q)$ lies in the higher stratum
$S_4(C(Y))$. Equivalently, there is no genuine codimension-two cone point in
$B_{2\delta}(q)$. By compactness of the quantitative stratification, after
shrinking $\delta$ and choosing $\eta>0$, all sufficiently small balls in
$B_\delta(q)$ which are singular are contained in the effective
$(2n-4)$-stratum $\mathcal S^{2n-4}_{\eta,\rho}(C(Y))$. Hence the Cheeger--Naber quantitative-stratification estimate
\cite{CheegerNaberQuant} gives, for all small $\rho$,
\[
        \operatorname{Vol}\bigl(T_\rho(S(C(Y)))\cap B_\delta(q)\bigr)
        \leq C\rho^{4-\eta}.
\],
Here the volume is the $2n$-dimensional Hausdorff measure of the Ricci-limit
space.

On the other hand, each $x_i\in \Delta_{X_\infty}^{smooth}\cap S'(X_\infty)$
has a neighborhood modeled on
$\mathbb C_{\beta_i}\times\mathbb C^{n-1}$, with
$\beta_i\leq 1-\frac1{N_{\mathcal F}}$. Since the cone angles belong to a finite
set, the local structure theorem gives a uniform lower tubular-volume bound:
there are constants $c>0$ and $\rho_0>0$, independent of $i$, such that for
$0<\rho<\rho_0$,
\[
        \operatorname{Vol}\bigl(T_\rho(\Delta_{X_\infty})\cap
        B_\delta(x_i)\bigr)
        \geq c\,\rho^2\delta^{2n-2}
\]
in the rescaled metric $r_i^{-2}d_\infty$.

Passing to the pointed Gromov--Hausdorff limit, the left-hand side is forced
into the $\rho$-neighborhood of the singular set of $C(Y)$ inside
$B_{2\delta}(q)$. Thus, for $i$ sufficiently large,
$c\,\rho^2\delta^{2n-2}\leq C\rho^{4-\eta}$. This is impossible for $\rho>0$
small, since $4-\eta>2$. Therefore $q\in\Delta_{C(Y)}$.

We now prove $(2)$. Again it suffices to show for $ q \in \Delta_{C(Y)}^{smooth} \cap S'Y$. The tangent cone at $q$ then splits locally as $\mathbb{C}\beta \times \mathbb{C}^{n-1}$. Let $D$ be a small transverse disk centered at $q$ intersecting $\Delta_{C(Y)}$ only at $q$; this can be done by the smoothness assumption, and let $\gamma = \partial D$.  The holonomy around $\gamma$ is $\exp(2 \pi i (1 - \beta)) \neq 1$. Assume for contradiction that no sequence $\Delta_{X_\infty}$ converges to $q$. Then there exists a radius $\delta > 0$ such that for all large $i$, the preimage of $B_\delta(q)$ in the rescaled space $(X_\infty, r_i^{-2} \omega_\infty)$ is entirely disjoint from $\Delta_{X_\infty}$. Which means (identify $B_\delta(q)$ with the preimage) that $B_\delta(q)$ only intersects $S_4$ or higher codimension strata of $X_\infty$. But observe that $S \backslash \Delta_{X_\infty}$ is a complex subvariety $W$ of complex codimension at least $2$. Thus $B_\delta(q) \cap S$ is contained in a complex analytic set of codimension at least $2$ (alternatively, one can use the Minkowski dimension bounds as proved in Cheeger-Jiang-Naber). Thus by standard Thom transversality for stratified spaces, a generic smooth map from a two dimensional disc to a $2n$-dimensional space would completely avoid $W$; thus we can perturb $D_i$ by an arbitrarily small amount such that it completely avoids $W$. Now that $D_i$ is fully regular; we compute: $Hol(\gamma_i) = \exp(\int_{D_i} Ric(r_i^{-2} \omega_\infty)) \leq \exp(C \lambda r_i^2) \rightarrow 1$. This contradicts, since by Anderson's theorem, $\partial D_i$ is contained in the smooth set of both the spaces so the holonomy should converge.
\end{proof}

\section{Analytic inputs from the metric cone}

Fix $p\in X_\infty$ throughout this section. We study the space
$\mathcal C_p$ of all tangent cones at $p$. Fix
$C(Y)\in\mathcal C_p$, realized by a sequence $r_i\to0$ of blow-ups
\[
(X_\infty,r_i^{-2}\omega_\infty,p)
\longrightarrow
(C(Y),\omega_{C(Y)},o).
\]
Let $\xi_0$ be the Reeb vector field of $\omega_{C(Y)}$, $Y$ be the link at
$\{r=1\}$, and let $T_\xi := \overline{\{\exp(t\xi_0):t\in \mathbb R\}} \subset \operatorname{Aut}(C(Y),\Delta)$ be the compact torus generated by the Reeb flow.

By \cite[Theorem 1 and Section 4]{SLII} (or equivalently following
the Donaldson--Sun construction in \cite[Section 2.3]{DSII}) the cone $C(Y)$
admits a $T_\xi$-equivariant affine embedding $\Psi:C(Y)\hookrightarrow \mathbb C^N$ by homogeneous holomorphic functions. Under this embedding, the homothetic
$\mathbb R_{>0}$-action, and hence the Reeb action, extends to a linear torus
action on $\mathbb C^N$.

Moreover, by \cite[Lemma 2.5]{HeSunFrankel}, as used in
\cite[Appendix, p.30]{DSII}, for every
$\xi\in (\mathbb R_{>0})^N$ there is a $T^N$-invariant K\"ahler cone metric
$\omega_\xi$ on $\mathbb C^N\setminus\{0\}$, with Reeb vector field $\xi$.
Restricting $\omega_\xi$ to $C(Y)$ gives a $T_\xi$-invariant K\"ahler cone
metric on $C(Y)$ whenever $\xi\in \mathfrak t_\xi^+ :=
\mathfrak t_\xi\cap (\mathbb R_{>0})^N$. We write $\omega_\xi=\frac{1}{4}dd^c r_\xi^2.$
For $\xi\in \mathfrak t_\xi^+$, define
\[
        V(\xi):=
        \int_{C(Y)}\exp\left(-\frac{r_\xi^2}{2}\right)
        (dd^c r_\xi^2)^n .
\]
Up to a dimensional normalization, this is the volume functional of the link.

We have the following variational formula.

\begin{lemma}[\cite{DSII}, Lemma 4.1]
Let $ d\mu_\xi:=\exp\left(-\frac{r_\xi^2}{2}\right) (dd^c r_\xi^2)^n, \eta_\xi:=d^c\log r_\xi$.
Then  $dV_\xi(\delta\xi) =  -n\int_{C(Y)}\eta_\xi(\delta\xi)\,d\mu_\xi$
and $\operatorname{Hess}V_\xi(\delta\xi,\delta'\xi)  =  n(n+1)\int_{C(Y)}  \eta_\xi(\delta\xi)\eta_\xi(\delta'\xi)\,d\mu_\xi$.

In particular, $V$ is strictly convex on $\mathfrak t_\xi^+$.
\end{lemma}

\begin{proof}[Sketch]
This is exactly the computation in \cite[Lemma 4.1]{DSII}. One writes
\[
        \delta(r_\xi^2)=r_\xi^2\varphi .
\]
Varying the identity $\mathcal L_{r\partial_r}r_\xi^2=2r_\xi^2$ gives
\[
        d^c\varphi(\xi)=-2\eta_\xi(\delta\xi).
\]
Then one differentiates
\[
        V(\xi)=\int_{C(Y)}
        e^{-r_\xi^2/2}(dd^c r_\xi^2)^n
\]
and integrates by parts on the normal affine cone. The same proof works in the
present log setting because the boundary divisor $\Delta$ is not used in this
calculation: the formula depends only on the normal affine variety, the
equivariant embedding, and the ambient $T^N$-invariant cone metrics. Singularities
are handled exactly as in the appendix of \cite{DSII}, by performing the
calculation on a $T_\xi$-equivariant log resolution and using that the exceptional
terms are integrable.
\end{proof}

By Theorem~\ref{thm:tangent-cone-log-structure},
$K_{C(Y)}+\Delta$ is $\mathbb Q$-Cartier. Choose $N$ such that $N(K_{C(Y)}+\Delta)$ is Cartier and trivial, and choose the
nowhere-vanishing log-canonical trivializing section
$s_\Delta\in H^0\left(C(Y), \mathcal O_{C(Y)}(N(K_{C(Y)}+\Delta))\right)$
constructed in the proof of Theorem~\ref{thm:tangent-cone-log-structure}.

The torus $T_\xi$ preserves both $C(Y)$ and $\Delta$, hence it acts on
$\mathcal O_{C(Y)}(N(K_{C(Y)}+\Delta))$. For $t\in T_\xi$, the section
$t^*s_\Delta$ is again a nowhere-vanishing trivialization. Therefore $t^*s_\Delta=u_t s_\Delta$ for some unit $u_t\in H^0(C(Y),\mathcal O_{C(Y)})^*$. Since $C(Y)$ is a normal
affine cone with vertex as the unique closed orbit, its only regular units are
constants. Hence $u_t\in\mathbb C^*$, and $s_\Delta$ is a $T_\xi$-eigensection.

Equivalently, there is a linear character $\chi_\Delta:\mathfrak t_\xi\rightarrow \mathbb R$  such that $\mathcal L_\xi s_\Delta=\sqrt{-1}\chi_\Delta(\xi)s_\Delta$.
In particular, the normalized Reeb hyperplane is $H_\Delta:=\{\xi\in\mathfrak t_\xi^+:\chi_\Delta(\xi)=nN\}.$

On $C(Y)\setminus\operatorname{Supp}(\Delta)$, note that the conical Ricci-flat equation is
equivalent to $ \omega_{C(Y)}^n = C\,(s_\Delta\wedge \bar s_\Delta)^{1/N}$
where $s_\Delta$ is viewed as a log-volume form. We now set up the appropriate functionals and functional spaces to do the variations. 

\begin{definition}
On $Y$, we have the data $(\xi_0,\eta,\omega^T)$, where
$\omega^T=\frac12 d\eta$ is the transverse K\"ahler structure. Let $\mathcal H$
be the space of bounded basic transverse K\"ahler potentials.

Let $\Omega_\Delta$ be the $T_\xi$-invariant measure on $Y$ defined by $dr\wedge \Omega_\Delta  = (s_\Delta\wedge \bar s_\Delta)^{1/N}$.
For any $\phi\in\mathcal H$, define the log Ding functional
\[
        D_\Delta(\phi)
        =
        I(\phi)-\log\int_Y e^{-\phi}\Omega_\Delta ,
\]
where
\[
        I(\phi)=
        -\frac{1}{nV(\xi)}
        \sum_{j=0}^{n-1}
        \int_Y
        \phi\,(d\eta)^j\wedge
        (d\eta+dd^c\phi)^{n-1-j}\wedge \eta .
\]
\end{definition}

 Observe that the Euler--Lagrange equation of $D_\Delta$ is $(d\eta+dd^c\phi)^{n-1}\wedge\eta =
 C e^{-\phi}\Omega_\Delta$; Equivalently, the cone metric $\omega_\phi = \frac14 dd^c(r^2e^\phi)$
is Ricci-flat with cone singularities along $\Delta$ and Reeb field $\xi_0$.

\begin{lemma}[Log Futaki variation]
Let $\delta\xi\in T_{\xi_0}H_\Delta$. Let $f_t$ be the one-parameter family of
holomorphic transformations generated by $J\delta\xi$, and let $\phi(t)$ be the
corresponding path of transverse K\"ahler potentials. Then $\frac{d}{dt}D_\Delta(\phi(t))  = -\int_{C(Y)}\mathcal L_{J\delta\xi}h_\Delta\,d\mu_{\xi_0}$
where  $h_\Delta := -\log \frac{(s_\Delta\wedge\bar s_\Delta)^{1/N}}  {\omega_{\xi_0}^n}$.
In particular, at the conical Ricci-flat metric, the log Futaki invariant vanishes,
and therefore $\xi_0$ is a critical point of $V|_{H_\Delta}$.
\end{lemma}

\begin{proof}[Sketch]
This is the log analogue of \cite[Equations (4.6)--(4.7)]{DSII}. The calculation
is the same as in \cite[Lemma 12]{CDSIII}: differentiating the Ding functional
along a holomorphic one-parameter subgroup gives the analytic log Futaki
invariant. The divisor contribution is exactly absorbed into the log-volume form
$s_\Delta$, or equivalently into the factor
$\prod_i |s_i|_{h_i}^{2(\beta_i-1)}$ if ordinary canonical frames are used.
Since the conical Ricci-flat metric is a critical point of $D_\Delta$, the log
Futaki invariant vanishes. Combining this with the Donaldson--Sun identity \cite[Equations~(4.6)--(4.7)]{DSII} $dV(\delta\xi)
  = -\frac12 \int_{C(Y)}\mathcal L_{J\delta\xi}h_\Delta\,d\mu_{\xi_0}$
for $\delta\xi\in T_{\xi_0}H_\Delta$, we obtain
$d(V|_{H_\Delta})(\delta\xi)=0$.
\end{proof}

We also record the cone version of the Bando--Mabuchi uniqueness theorem,
which will be used in the next section to prove uniqueness of the tangent cone.

\begin{lemma}[Uniqueness of conic RF metrics]
Given the data $(C(Y),\Delta,\xi_0)$, a Ricci-flat K\"ahler cone metric with cone
singularities along $\Delta$ and Reeb field $\xi_0$ is unique up to the identity
component of
\[
        \operatorname{Aut}(C(Y),\Delta,\xi_0).
\]
\end{lemma}

\begin{proof}[Sketch]
A conical Ricci-flat cone metric with fixed Reeb field
corresponds to a bounded critical point of the log Ding functional
$D_\Delta$ on the space $\mathcal H$ of bounded basic transverse K\"ahler
potentials. By the Berndtsson convexity theorem, in the singular/conical form
used in \cite[Appendix 1]{CDSIII}, the Ding functional is convex along bounded
geodesics. Therefore two critical points are joined by a geodesic along which
$D_\Delta$ is affine. The equality case in Berndtsson's theorem produces a
holomorphic vector field preserving the log data; its flow pulls one metric back
to the other. This gives uniqueness modulo
$\operatorname{Aut}(C(Y),\Delta,\xi_0)^0$.

For $\Delta=0$ on a possibly singular affine cone, this is
\cite[Proposition~4.8]{DSII}. The corresponding Ding-functional uniqueness
argument for compact weak conical K\"ahler--Einstein pairs is explained in
\cite[Appendix~1]{CDSIII}, while \cite[Proposition~5]{DatarSzekelyhidi} gives a
related twisted uniqueness statement on singular $\mathbb Q$-Fano spaces. In
the present logarithmic cone setting, one applies the same convexity and
equality-case argument to the log-volume form $s_\Delta$; the klt condition gives
the required $L^p$ integrability on an equivariant log resolution.
\end{proof}

We next show that the Reeb vector has algebraic coefficients. This algebraicity
will provide the discreteness input used in proving that the tangent cones in
$\mathcal C_p$ have the same Hilbert function.

As in Donaldson-Sun \cite[Proposition~2.21]{DSII}, algebraicity of the Reeb
vector supplies the discreteness input for the holomorphic spectrum. In our
logarithmic setting the normalization hyperplane is instead determined by the
log-canonical eigensection $s_\Delta$, so the character records the data of the
pair $(C(Y),\Delta)$ instead.

\begin{theorem}
$\xi_0$ is an algebraic vector.
\end{theorem}

\begin{proof}
Let $\mathfrak t^+$ be the Reeb cone. Since $s_\Delta$ is a $T$-eigensection, any $\xi\in\mathfrak t$ acts on $s_\Delta$ by a character. Thus define $\chi_\Delta:\mathfrak t\to\mathbb R$ by $\mathcal L_\xi s_\Delta=\sqrt{-1}\chi_\Delta(\xi)s_\Delta$. It is clear that $\chi_\Delta$ is linear. Define $H_\Delta:=\{\xi\in\mathfrak t^+:\chi_\Delta(\xi)=nN\}$; thus $H_\Delta$ is a hyperplane section of $\mathfrak t^+$.

Note that $H_\Delta$ has rational coefficients. Under the $T$-equivariant embedding $C(Y)\hookrightarrow\mathbb C^M$, the torus $T_\mathbb C$ is an algebraic subtorus of $(\mathbb C^*)^M$. Hence its character lattice is obtained from the standard lattice $\mathbb Z^M$. Since $s_\Delta$ spans a one-dimensional $T_\mathbb C$-representation of the linearized line bundle $\mathcal O(N(K_{C(Y)}+\Delta))$, its weight is an integral character of $T_\mathbb C$. In coordinates on $\mathfrak t\subset \mathbb R^M$, this means $\chi_\Delta(\xi)=\langle m,\xi\rangle$ for some $m\in \mathbb Q^M$, in fact integral after the fixed index $N$ is chosen. Therefore $H_\Delta$ is defined by a rational linear equation.

Now by Martelli--Sparks--Yau \cite[Section~5]{MartelliSparksYau}, in the case
where $C(Y)\setminus\{0\}$ has smooth link, and by Collins--Sz\'ekelyhidi
\cite[Section~2]{CollinsSzekelyhidi} for polarized normal affine cones, the
normalized volume $V(\xi)/V_n$ is a rational function with rational coefficients
in the components of $\xi$. The proof uses the index character $F(\xi,t)=\sum_{\alpha} e^{-t\langle\alpha,\xi\rangle}\dim H_\alpha$ whose expansion near $t=0$ has the form $F(\xi,t)=a_0(\xi)(n-1)!t^{-n}+a_1(\xi)(n-2)!t^{1-n}+\cdots$
with $a_0(\xi)$ a rational function with rational coefficients. Moreover
$a_0(\xi)=c_nV(\xi)$ for a universal dimensional constant $c_n$.

Since $H_\Delta$ is defined over $\mathbb Q$, the restricted volume functional $V|_{H_\Delta}$ is also a rational function with rational coefficients. By the variational formula established before, and by the log Futaki vanishing for the conical Ricci-flat metric, $\xi_0$ is a critical point of $V|_{H_\Delta}$. Moreover the Hessian formula gives $\operatorname{Hess}V(\delta\xi,\delta\xi)=n(n+1)\int_{C(Y)}\eta(\delta\xi)^2\,d\mu$, so the Hessian is positive definite on $H_\Delta$. Hence $\xi_0$ is an isolated critical point of $V|_{H_\Delta}$.

After clearing denominators, the critical point equations for $V|_{H_\Delta}$ become a system of polynomial equations with rational coefficients. By \cite[Lemma 4.6]{DSII}, an isolated zero of a system of polynomial equations with rational coefficients has algebraic coordinates. Therefore the components of $\xi_0$ are algebraic numbers.
\end{proof}

The following $L^2-L^\infty$ and gradient estimate will be used repeatedly throughout the rest of this section.

\begin{lemma}[Local estimates for weakly holomorphic functions]
\label{lem:local-holomorphic-estimates}
There is a constant $C$, depending only on the dimension and the non-collapsing
constant, such that the following holds. Let $Z$ be either $X_\infty$ or a
metric tangent cone $C(Y)$ of $X_\infty$. If $f$ is weakly holomorphic on
$B_Z(x,2r)$, meaning $f$ is holomorphic on the regular locus and locally $L^\infty$
across the singular set, then
\[
        \|f\|_{L^\infty(B_Z(x,r))}
        \leq
        C r^{-n}\|f\|_{L^2(B_Z(x,2r))},
        \qquad
        \operatorname{Lip}_{B_Z(x,r)}(f)
        \leq
        C r^{-n-1}\|f\|_{L^2(B_Z(x,2r))}.
\]
\end{lemma}

\begin{proof}
On the regular locus, a holomorphic function is harmonic. Here the full metric
singular set may contain the real codimension-two conical stratum; only the
residual higher stratum $S_4$ has real codimension at least four. The cutoff
argument treats $S_2\setminus S_4$ using the local conic charts and uses the
codimension-four smallness only for the remaining higher stratum.
Since $f$ is locally
bounded and weakly holomorphic across the analytic singular set, the cutoff
functions constructed for lower-Ricci K\"ahler limits imply that $f$ is weakly
harmonic on the whole ball. The required cutoffs are supplied either by the good-cutoff construction in the
lower-Ricci K\"ahler setting \cite[Section~3]{SLI}, or by the RCD/Sobolev
theory for singular K\"ahler--Einstein metrics developed in
\cite{SzekelyhidiRCD,GuoPhongSongSturmSobolev}.

For tangent cones, this harmonicity and the resulting estimates are precisely
\cite[Proposition~10]{SLII}: if $\bar\partial f=0$ on
$\mathcal R_\epsilon$ and $f$ has polynomial $L^2$ growth, then $f$ is harmonic
on the whole tangent cone and satisfies the local $L^\infty$ and Lipschitz
bounds. The proof uses heat-flow regularization together with Ding's heat-kernel
convergence theorem.

For $X_\infty$, the polarized approximation, equivalently the singular
K\"ahler--Einstein/RCD theorem of Sz\'ekelyhidi, gives a noncollapsed RCD space.
Weakly harmonic functions on noncollapsed RCD spaces satisfy the standard local
mean-value and Lipschitz estimates. Applying these estimates at unit scale and
rescaling gives the factors $r^{-n}$ and $r^{-n-1}$, since the real dimension is
$2n$.
\end{proof}

Unlike the Donaldson--Sun setting, we have only a lower Ricci bound rather than
a two-sided Ricci bound. We therefore need sharper control of the metric
singular set in order to obtain the required cutoff estimates.

The smallness near the singular set can be obtained either from the good cutoff
construction in \cite[Appendix, Proposition 5.1]{SLI}, or more directly
from the Cheeger--Jiang--Naber bound on the codimension-two Minkowski content
of $X\setminus R_\epsilon$, as used in \cite[Section 3]{SLI}. We use the
appendix route here because it avoids importing the full quantitative
stratification theorem into the proof.

\begin{lemma}[Grafting lemma]
Suppose $(C(Y_i),o)\rightarrow (C(Y),o)$ in $\mathcal C_p$. Let $B$ be a ball in $C(Y)$, and $f$ is holomorphic on $B$. Then shrinking $B$ to $B'$ if necessary, we can find holomorphic functions $f_i$ on $B_i\subset C(Y_i)$ such that $(C(Y_i),B_i,o)\rightarrow (C(Y),B,o)$ and $f_i\rightarrow f$ on $B'$.
\end{lemma}

\begin{proof}
Choose relatively compact balls $B''\Subset B'\Subset D\Subset B$. Let $\Sigma\subset C(Y)$ and $\Sigma_i\subset C(Y_i)$ be the metric singular sets. Fix $\kappa_i\to0$. By \cite[Appendix, Proposition 5.1]{SLI}, applied to $\overline D\subset C(Y)$ and to the corresponding compact sets $\overline D_i\subset C(Y_i)$, we may choose Lipschitz functions $\chi_i$ on $C(Y_i)$ such that $\chi_i=1$ on a neighborhood of $\overline D_i\cap\Sigma_i$, $\operatorname{supp}\chi_i$ is contained in the $\kappa_i$-neighborhood of $\overline D_i\cap\Sigma_i$, and $\int_{D_i}|\nabla\chi_i|^2<\kappa_i$. Set $\beta_i:=1-\chi_i$.

On compact subsets of $D\setminus\Sigma$, the convergence $C(Y_i)\to C(Y)$ is smooth in the sense of the holomorphic charts of \cite[Theorem 1.4 and Proposition 2.4]{SLI}. Thus, after identifying compact subsets of $D\setminus\Sigma$ with their images in $D_i\setminus\Sigma_i$, the function $f$ gives smooth functions $\tilde f_i$ on $D_i\setminus\Sigma_i$ such that $\tilde f_i\to f$ in $C^1_{\mathrm{loc}}(D\setminus\Sigma)$ and $|\bar\partial_i\tilde f_i|\to0$ uniformly on compact subsets of $D\setminus\Sigma$. Here $\bar\partial_i$ denotes the complex structure on $C(Y_i)$.

Choose a radial cutoff $\theta_i$ supported in $D_i$ and equal to $1$ on $B'_i$, with $|\nabla\theta_i|\le C$. Define $\sigma_i:=\theta_i\beta_i\tilde f_i$ on $R_\epsilon(C(Y_i))$. Then $\sigma_i$ is smooth and compactly supported in $R_\epsilon(C(Y_i))$, because $\beta_i$ vanishes on a neighborhood of $\Sigma_i$. Moreover $\sigma_i$ agrees with $\tilde f_i$ on $B'_i$ away from the $\kappa_i$-neighborhood of $\Sigma_i$.

Let $\alpha_i:=\bar\partial_i\sigma_i$. We have $\alpha_i=(\bar\partial_i\theta_i)\beta_i\tilde f_i+\theta_i(\bar\partial_i\beta_i)\tilde f_i+\theta_i\beta_i\bar\partial_i\tilde f_i$. The second term is small because $\int|\nabla\beta_i|^2=\int|\nabla\chi_i|^2<\kappa_i$. The third term is small because $\bar\partial_i\tilde f_i\to0$ on the region away from $\Sigma_i$. The first term is supported in $D_i\setminus B'_i$.

We now solve $\bar\partial_i u_i=\alpha_i$ on $R_\epsilon(C(Y_i))$ using \cite[Proposition 5]{SLII}. For the weight, take $\varphi_i=A_i w(r_i^2)$, where $w'>0$, $w''\le0$, and $w'(t)+tw''(t)>0$, as in \cite[equation (3.1)]{SLII}. The estimate gives
$\int_{R_\epsilon(C(Y_i))}|u_i|^2e^{-\varphi_i}\,d\mu_i\le \int_{R_\epsilon(C(Y_i))}\frac{|\alpha_i|^2}{A_i(w'(r_i^2)+r_i^2w''(r_i^2))}e^{-\varphi_i}\,d\mu_i$.

Choose $A_i\to\infty$ slowly. Since the $\bar\partial\theta_i$ term is supported a definite positive distance away from $B''_i$, the weight makes its contribution to $\int_{B''_i}|u_i|^2$ tend to zero. The $\bar\partial\beta_i$ term is bounded by $C\kappa_i^{1/2}$, and the $\bar\partial\tilde f_i$ term is $o_i(1)$. Thus $\int_{B''_i}|u_i|^2\,d\mu_i\to0$.

Set $f_i:=\sigma_i-u_i$. Then $\bar\partial_i f_i=0$ on $R_\epsilon(C(Y_i))$. By \cite[Proposition 10]{SLII}, $f_i$ is harmonic across the singular set and hence defines a holomorphic function on the normal analytic cone $C(Y_i)$.

It remains to prove convergence. On $B''_i$, $\sigma_i-\tilde f_i$ is supported in the $\kappa_i$-neighborhood of $\Sigma_i$, whose volume tends to zero; since $\tilde f_i$ is uniformly bounded, $\|\sigma_i-\tilde f_i\|_{L^2(B''_i)}\to0$. Also $\|\tilde f_i-f\|_{L^2(B''_i)}\to0$ under the smooth convergence away from $\Sigma$, and $\|u_i\|_{L^2(B''_i)}\to0$. Hence $\|f_i-f\|_{L^2(B''_i)}\to0$.

The estimation lemma gives uniform $L^\infty$ and Lipschitz bounds for $f_i$ on $B'''_i\Subset B''_i$. Therefore every subsequence has a further subsequence converging uniformly on compact subsets to a holomorphic function $g$ on $B'''$. The $L^2$ convergence forces $g=f$. Hence the whole sequence converges to $f$ uniformly on compact subsets of $B'''$. Relabeling $B'''$ as $B'$ proves the lemma.
\end{proof}

This is enough to conclude the rigidity of the spectrum. More precisely, we have:

\begin{theorem}
For any $C(Y) \in \mathcal C_p$, the holomorphic spectrum
$S(C(Y)) = \{ d: \exists f \in R(C(Y)), \operatorname{charge}(f)=d \}$
is independent of $C(Y) \in \mathcal C_p$. Also the Hilbert function
$d \mapsto \dim R_d(C(Y))$ is independent of $C(Y)$ as well.
\end{theorem}

\begin{proof}
It is standard that for any metric space, the space of possible tangent cones at a
point is compact and connected; in our setting this is \cite[Lemma 3.2]{DSII}.
Fix $C(Y)\in\mathcal C_p$, and choose $D\notin S(C(Y))$. Let
$E_D(C(Y))=\oplus_{0<d<D}R_d(C(Y))$.

Let $C(Y_i)\rightarrow C(Y)$. We show that $\dim E_D(C(Y_i))$ is eventually
constant. First let $\{f^i_1,\dots,f^i_{m_i}\}$ be an orthonormal basis of
$E_D(C(Y_i))$ on the unit ball. Since the functions are homogeneous with charges
in $(0,D)$, their $L^2$ norms on fixed larger balls are uniformly controlled by
their $L^2$ norms on the unit ball. By the estimation lemma, after passing to a
subsequence the $f^i_a$ converge locally uniformly and weakly in $L^2$ to
holomorphic functions $f_a$ on $C(Y)$. The limit functions are still homogeneous
of charge at most $D$, and orthonormality is preserved by strong convergence of
$L^2$ norms. Hence $\limsup_i \dim E_D(C(Y_i))\leq \dim E_D(C(Y))$.

For the opposite inequality, suppose $\dim E_D(C(Y_i))<\dim E_D(C(Y))$ along a
subsequence. Passing to a further subsequence, an orthonormal basis of
$E_D(C(Y_i))$ converges to an orthonormal basis of a proper subspace
$E'\subset E_D(C(Y))$. Choose $f\in E_D(C(Y))$ with $\|f\|_{L^2(B)}=1$ and
$f\perp E'$. By the grafting lemma, after shrinking the ball, there are holomorphic
functions $f_i$ on $C(Y_i)$ with $f_i\rightarrow f$ locally uniformly.

Write $f_i=g_i+h_i$, where $g_i\in E_D(C(Y_i))$ and $h_i\perp E_D(C(Y_i))$ in
$L^2$ on the unit ball. Since $f\perp E'$, we get $\|g_i\|_{L^2}\rightarrow0$.
On the other hand, $h_i$ has no homogeneous component of charge $<D$. Therefore
the three-annulus estimate gives a uniform lower growth bound for $h_i$ from the
unit ball to a slightly larger ball. But $f_i\rightarrow f\in E_D(C(Y))$, whose
growth is strictly less than this bound because $D\notin S(C(Y))$. This is a
contradiction. Hence $\dim E_D(C(Y_i))=\dim E_D(C(Y))$ for all large $i$.

Thus $\dim E_D(C(Y))$ is locally constant on $\mathcal C_p$. Since
$\mathcal C_p$ is connected, it is constant. Varying $D$ over the complement of
the discrete spectrum gives both the independence of the spectrum and the
independence of the Hilbert function.
\end{proof}

We now return to $X_\infty$. Fix $p \in X_\infty$ and $\lambda=2$. Our goal
is to construct a valuation $\nu$ on $\mathcal O_{X_\infty,p}$. We will only be
able to show it is a semivaluation for now; we will conclude it is a genuine
valuation in the next section after we upgrade the weak limit to the strong limit.

Let $Z_i$ be the rescaling of $(X_\infty,p)$ by the factor $\lambda^i$, and let
$B_i$ be the unit ball centered at $p$ in $Z_i$. Thus $B_i$ corresponds to the
ball $B(p,\lambda^{-i})$ in the original metric. For $0\neq f\in
\mathcal O_{X_\infty,p}$, represented on a neighborhood of $p$, set
$\|f\|_i=(\int_{B_i}|f|^2\,d\mu_i)^{1/2}$ and $[f]_i=f/\|f\|_i$.

\begin{theorem}[Metric order and semivaluation]
The limit
$$\nu(f):=-(\log\lambda)^{-1}\lim_{i\to\infty}\log(\|f\|_{i+1}/\|f\|_i)$$
exists and lies in $S\cup\{+\infty\}$, where $S$ is the common holomorphic
spectrum of the tangent cones at $p$. Moreover:
\begin{enumerate}
    \item if $\nu(f)=+\infty$, then $[f]_i$ converges weakly by sequence to $0$
    on every tangent-cone subsequence;
    \item if $\nu(f)=d<+\infty$, then $[f]_i$ converges strongly by sequence to
    a nonzero homogeneous holomorphic function of charge $d$ on the resulting
    tangent cone;
    \item $\nu$ is a semivaluation:
    $\nu(f+g)\geq \min\{\nu(f),\nu(g)\}$ and
    $\nu(fg)=\nu(f)+\nu(g)$.
\end{enumerate}
Consequently $I_\infty:=\{f\in\mathcal O_{X_\infty,p}:\nu(f)=+\infty\}$ is an
ideal, and $\nu$ descends to a genuine valuation on
$\mathcal O_{X_\infty,p}/I_\infty$.
\end{theorem}

\begin{proof}
We use the three-annulus inequality in the form of \cite[Lemma 3.6]{DSII} adapted to our setting. Let
$B=\{r<1\}$ in a tangent cone $C(Y)$, and let $\Lambda$ denote dilation by
$\lambda$. If $F$ is $L^2$ and holomorphic on $B$, write
$F=\sum_{d\in S}F_d$ into homogeneous charge pieces. Then
$\Lambda.F=\sum_{d\in S}\lambda^dF_d$ and
$\Lambda^2.F=\sum_{d\in S}\lambda^{2d}F_d$. Since distinct charge spaces are
orthogonal in $L^2(B)$, Cauchy--Schwarz gives
$\|\Lambda.F\|_{L^2(B)}^2\leq \|F\|_{L^2(B)}\|\Lambda^2.F\|_{L^2(B)}$.
Note, in particular, that equality holds if and only if all nonzero $F_d$ have the same charge, i.e. if and
only if $F$ is homogeneous.

We next use the corresponding crossing lemma. Fix $\bar d\notin S$. Then there
is $i_0=i_0(\bar d)$ such that for all $j>i\geq i_0$ and every nonzero holomorphic
function $h$ defined on $B_i$, if $\|h\|_{i+1}\geq \lambda^{-\bar d}\|h\|_i$,
then $\|h\|_{j+1}>\lambda^{-\bar d}\|h\|_j$. Indeed, if this failed, after
normalizing by $\|h_\alpha\|_{\alpha+1}=1$ and passing to a tangent-cone limit,
the estimation lemma gives locally uniform convergence to a nonzero holomorphic
function $H$ satisfying
$\|H\|_{L^2(B)}\leq \lambda^{\bar d}$,
$\|\Lambda.H\|_{L^2(B)}=1$, and
$\|\Lambda^2.H\|_{L^2(B)}\geq \lambda^{-\bar d}$. These inequalities force
equality in the three-annulus inequality, so $H$ is homogeneous of charge
$\bar d$, contradicting $\bar d\notin S$.

Apply the crossing lemma to
$a_i(f):=-(\log\lambda)^{-1}\log(\|f\|_{i+1}/\|f\|_i)$. For every
$\bar d\notin S$, the sequence $a_i(f)$ eventually stays on one side of
$\bar d$. Since $S$ is discrete, this implies that $a_i(f)$ has a limit in
$S\cup\{+\infty\}$. This proves the existence of $\nu(f)$.

The convergence statements are the same as \cite[Corollary 3.8]{DSII}. If
$\nu(f)=+\infty$, any nonzero weak limit of $[f]_i$ would have finite homogeneous
charge by the three-annulus argument, contradicting $a_i(f)\to+\infty$. If
$\nu(f)=d<+\infty$, then on every tangent-cone subsequence, any weak limit
$F$ of $[f]_i$ satisfies
$\lambda^{-2d}\|\Lambda^2.F\|_{L^2(B)}
=\lambda^{-d}\|\Lambda.F\|_{L^2(B)}
=\|F\|_{L^2(B)}$.
Thus equality holds in the three-annulus inequality, and $F$ is homogeneous of
charge $d$. The estimation lemma gives uniform Lipschitz bounds on smaller balls,
and the convergence of the $L^2$ norms gives strong convergence.

It remains to prove the semivaluation properties. By the local $L^2$--sup
estimate, the above $L^2$ definition agrees with the sup-norm characterization
$\nu(h)=\lim_{r\to0}(\log r)^{-1}\log\sup_{B(p,r)}|h|$. Therefore
$\sup_{B(p,r)}|f+g|\leq \sup_{B(p,r)}|f|+\sup_{B(p,r)}|g|$ gives
$\nu(f+g)\geq \min\{\nu(f),\nu(g)\}$.

Similarly, $\sup_{B(p,r)}|fg|\leq \sup_{B(p,r)}|f|\sup_{B(p,r)}|g|$ gives
$\nu(fg)\geq \nu(f)+\nu(g)$. For the reverse inequality, first assume
$\nu(f)=d_f<+\infty$ and $\nu(g)=d_g<+\infty$. Choose a tangent-cone subsequence
along which $[f]_i$ and $[g]_i$ converge strongly to nonzero homogeneous
holomorphic functions $F$ and $G$ of charges $d_f$ and $d_g$. The tangent cone is
normal and irreducible, hence its holomorphic function ring is a domain, so
$FG\neq0$. Choose a regular point $x$ with $F(x)G(x)\neq0$, and choose
$x_i\in B_i$ converging to $x$. Then for all large $i$,
$|[f]_i(x_i)[g]_i(x_i)|\geq c>0$. Hence
$\sup_{B_i}|fg|\geq c\|f\|_i\|g\|_i$. Using the sup-norm characterization gives
$\nu(fg)\leq d_f+d_g$. Therefore $\nu(fg)=\nu(f)+\nu(g)$ in the finite case.
If one of $\nu(f),\nu(g)$ is $+\infty$, the inequality
$\nu(fg)\geq\nu(f)+\nu(g)$ already gives $\nu(fg)=+\infty$.

Hence $\nu$ is a semivaluation.
\end{proof}

\section{Filtration and degenerations}

In this section we construct the analogue of the Donaldson--Sun intermediate degeneration
\cite[Section~3]{DSII}
\[
W=\operatorname{Spec}\operatorname{gr}_\nu\mathcal O_{X_\infty,p},
\]
and track the log boundary through this first degeneration. We then introduce
an incidence space $\mathcal P$ in which $(W,\Delta_W)$ and
$(C(Y),\Delta_C)$ lie in a common equivariant parameter space.

Let $S \subset \mathbb R_{\geq 0}$ be the holomorphic spectrum of a tangent cone
at $p$. By the holomorphic rigidity theorem as proven before, both $S$ and the
Hilbert function $h:S\rightarrow\mathbb Z$ given by
$h(d)=\mu_d:=\dim R_d(C(Y))$ are independent of the tangent cone
$C(Y)\in\mathcal C_p$.

Let $T$ be the compact torus in $\operatorname{Aut}(C(Y))$ which is obtained as
the closure of the one-parameter flow generated by the Reeb flow. We can then
decompose $R(C(Y))$ as the $T$-weight spaces:
$R(C(Y))=\bigoplus_{\chi\in X^*(T)}R_\chi(C(Y))$. The Reeb vector field $\xi$
then assigns to each $T$-weight $\chi$ the charge $\langle \chi,\xi\rangle$,
thus $R_d(C(Y))=\bigoplus_{\langle\chi,\xi\rangle=d}R_\chi(C(Y))$.

By the affine algebraicity theorem, $C(Y)$ is a normal affine variety; hence its
coordinate ring $R(C(Y))$ is finitely generated. In the Donaldson--Sun setting
this is contained in \cite[Theorem~1.2]{DSII}, and under the Ricci lower bound
hypotheses used here it follows from \cite[Theorem~1.1]{SLII}. Thus pick $D$ such that
$R(C(Y))$ is generated by $E_D(C(Y)):=\bigoplus_{d\leq D}R_d(C(Y))$. Choosing
a $T$-homogeneous basis of $E_D(C(Y))$ gives us a $T$-equivariant affine embedding
$\Phi:C(Y)\rightarrow\mathbb C^N$. Under this embedding, the pushforward
$\Phi_*(\xi_{C(Y)})$ is a diagonal vector field on $\mathbb C^N$, with weights
coming from the Reeb charges.

The discreteness of $S$ also follows at this point. Indeed, choose homogeneous
generators $x_1,\dots,x_N$ of $R(C(Y))$ with positive charges $w_1,\dots,w_N$.
Then every charge occurring in $R(C(Y))$ lies in the additive semigroup
$\sum_j\mathbb Z_{\geq0}w_j$. Since each $w_j>0$, this semigroup has only
finitely many elements in every bounded interval $[0,M]$. Hence
$S=\{0=d_0<d_1<d_2<\cdots\}$ is discrete.

For $k\geq0$ define $I_k:=\{f\in A:\nu(f)\geq d_k\}$ and $I_{>d_k}:=\{f\in A:\nu(f)>d_k\}$.
Since there is no spectral value strictly between $d_k$ and $d_{k+1}$, we have
$I_{>d_k}=I_{k+1}$. Equivalently $I_k/I_{k+1}$ is the graded piece of charge
$d_k$. The semivaluation property gives $I_jI_k\subset I_\ell$ whenever
$d_\ell\leq d_j+d_k$.

\begin{definition}
Let $R_p:=\operatorname{gr}_\nu A:=\bigoplus_{k\geq0}I_k/I_{k+1}$, and let
$W:=\operatorname{Spec}R_p$.
\end{definition}

Observe that $W$ is an affine scheme graded by the semigroup $S$. After finite
generation is proved below, $W$ is an affine variety.

We recall the adapted-basis input from \cite[Section~3.2]{DSII}. Let
$A=\mathcal O_{X_\infty,p}$, and let $P\subset A$ be a finite-dimensional
subspace. 

Write $B_i$ for the unit ball in the rescaling of $(X_\infty,p)$ by
$\lambda^i$, so $B_i$ is the original ball $B(p,\lambda^{-i})$. For
$f\in P$, let $\|f\|_i:=\|f\|_{L^2(B_i)}$. The comparison map
$\Lambda_i:P|_{B_{i-1}}\to P|_{B_i}$ is induced by restricting the same germ to
the next rescaled ball. On tangent-cone limits, $\Lambda_i$ becomes the genuine
dilation action.

\begin{definition}[Adapted sequence of bases]
Let $P\subset A$ be finite-dimensional, with $\dim P=m$. An \emph{adapted
sequence of bases} is a sequence of bases
$\{G_i^1,\dots,G_i^m\}$ of $P$ such that:

\begin{enumerate}
\item $\|G_i^a\|_i=1$ for all $a$, and
$\int_{B_i}G_i^a\overline{G_i^b}\to0$ for $a\neq b$.

\item For each $a$,
$\Lambda_i(G_{i-1}^a)=\mu_{ia}G_i^a+p_i^a$, where
$p_i^a\in\operatorname{Span}\{G_i^1,\dots,G_i^{a-1}\}$ and
$\|p_i^a\|_i\to0$.

\item There are charges $d_1\leq\cdots\leq d_m$ in the holomorphic spectrum
such that, with our convention $B_i=B(p,\lambda^{-i})$,
$\mu_{ia}\to\lambda^{-d_a}$. Moreover the triangular error can be chosen inside
the equal-charge block:
$p_i^a\in\operatorname{Span}\{G_i^b:b<a,\ d_b=d_a\}$.
\end{enumerate}
\end{definition}

We now have the following results on the adapted complements and finite generation
of the associated graded ring from \cite[Section~3.2]{DSII}. We summarize and adapt the proof here for reference. We will need the weight decomposition in the adapted components part in the
subsequent sections where we keep track of how the cycles converge in the limit.

\begin{lemma}[Adapted complements]
For every $k\geq0$ there is a decomposition $I_k=I_{k+1}\oplus J_k$ such that
$\dim J_k=\mu_k:=\dim R_{d_k}(C(Y))$, and $J_k$ admits an adapted sequence of
bases whose tangent-cone limits form a basis of $R_{d_k}(C(Y))$.
\end{lemma}

\begin{proof}
This is the analogue of \cite[Proposition 3.14]{DSII}. The proof uses only the
analytic inputs established above: the metric order, the crossing lemma, the
grafting lemma, and the local $L^2$ estimates.

We proceed by induction on $k$. For $k=0$, take $J_0=\mathbb C$. Suppose the
result is known for all $j<k$. Let $\mathcal J_k$ be the set of finite-dimensional
subspaces $J\subset I_k$ such that $J\cap I_{k+1}=0$ and $J$ admits an adapted
sequence of bases with degree set $\{d_k,\dots,d_k\}$. By the convergence of
adapted bases to homogeneous functions, every such $J$ has $\dim J\leq\mu_k$.
Choose $J_k$ maximal.

We claim $\dim J_k=\mu_k$. If not, take an adapted basis of $J_k$ and pass to a
tangent-cone limit. Its limits span a proper subspace of $R_{d_k}(C(Y))$. Choose
$f\in R_{d_k}(C(Y))$ orthogonal to this span. By the grafting lemma, $f$ is
approximated by holomorphic functions $f_i$ on the rescaled balls. The crossing
lemma implies that, for $i$ large, the metric order of $f_i$ is at most $d_k$,
and after projecting away the previously constructed pieces $\oplus_{j<k}J_j$,
the relative order is still at most $d_k$. Therefore \cite[Proposition~3.12]{DSII}
produces an adapted basis for $J_k\oplus\mathbb C\{F\}$ with degree $d_k$ for
some fixed grafted function $F\in I_k$. Also $F\notin I_{k+1}+J_k$, because its
tangent-cone limit has a nonzero component orthogonal to the limits of $J_k$.
This contradicts maximality. Hence $\dim J_k=\mu_k$.

It remains to show $I_k=I_{k+1}\oplus J_k$. Let $f\in I_k$. After subtracting an
element of $J_k$, we may assume that $f$ is $L^2$-orthogonal to $J_k$ on a large
rescaled ball. If the image of $f$ in $I_k/I_{k+1}$ were nonzero, the same
relative-order argument would produce one more adapted degree-$d_k$ element
orthogonal to $J_k$, contradicting $\dim J_k=\mu_k$. Hence the remainder lies in
$I_{k+1}$, proving the decomposition.
\end{proof}

\begin{lemma}[Finite generation of the associated graded ring]
There exists $D\gg1$ such that $R_p=\operatorname{gr}_\nu A$ is generated by $\bigoplus_{d_k\leq D}I_k/I_{k+1}$. Consequently $W=\operatorname{Spec}R_p$ admits a $T$-equivariant affine embedding into the same weight-decomposed vector space as $C(Y)$.
\end{lemma}

\begin{proof}
Choose $D$ so that $R(C(Y))$ is generated by
$E_D(C(Y))=\bigoplus_{d_k\leq D}R_{d_k}(C(Y))$, and set
$k_0=\max\{k:d_k\leq D\}$. Choose adapted bases for $P:=\bigoplus_{0<k\leq k_0}J_k$.
These bases define holomorphic maps $F_i:B_i\to\mathbb C^N$ converging to the
homogeneous embedding $\Phi:C(Y)\hookrightarrow\mathbb C^N$.

For each $k$, let $S_k$ be the space of weighted homogeneous polynomials on
$\mathbb C^N$ of charge $d_k$, and let $V_k:=\ker\bigl(S_k\to R_{d_k}(C(Y))\bigr)$.
Choose a complement $S_k=V_k\oplus Q_k$, so $Q_k\simeq R_{d_k}(C(Y))$.
Let $T_{k,i}\subset A$ be the pullback of $Q_k$ by $F_i$. Then, for $i$ large, by similar adapted basis arguments, we get $I_k=T_{k,i}\oplus I_{k+1}$. Since $R(C(Y))$ is generated in charges $\leq D$, the
same splitting implies that every graded piece $I_k/I_{k+1}$ of $R_p$ is
generated by the pieces with $d_k\leq D$. Hence $R_p$ is finitely generated.
\end{proof}

Now, let $G_\xi$ be the reductive subgroup of $GL(\mathbb{C}^N)$ which preserves the $T$-weight decomposition. We can think of it more concretely as let $\mathbb{C}^N = \bigoplus_{\chi} V_\chi$ be the $T$-weight decomposition determined by the embedding; then $G_\xi = \prod_{\chi} GL(V_\chi)$. 

Let $\operatorname{Hilb}_h^T$ be the multigraded Hilbert scheme parametrizing $T$-invariant affine subschemes of $\mathbb C^N$ with the fixed Hilbert function $h(d)=\mu_d.$ The action of $G_\xi$ thus also induces actions on $\mathbb{CP}^N$, and $\operatorname{Hilb}_h^T$. 

The $T$-grading is positive in the sense of
\cite[Corollary~1.2]{HaimanSturmfels}. Indeed, the coordinate functions of the
chosen homogeneous embedding have strictly positive Reeb charges. If a
nonconstant monomial had $T$-weight zero, pairing its weight with the Reeb
vector $\xi$ would give zero; on the other hand this pairing is the sum of its
nonnegative exponents times the strictly positive coordinate charges, hence is
strictly positive. Thus the only monomial of $T$-weight zero is $1$.
Consequently \cite[Corollary~1.2]{HaimanSturmfels} implies that
$\operatorname{Hilb}_h^T$ is projective.

\begin{lemma}[Ambient Donaldson--Sun degeneration]
\label{lem:ambient-dsii-degeneration}
The affine cones $W=\operatorname{Spec}\operatorname{gr}_\nu A$ and $C(Y)$
define points $[W],[C(Y)]\in\operatorname{Hilb}_h^T$. Moreover, after choosing
the Donaldson--Sun adapted bases \cite[Section~3.2]{DSII} attached to the tangent-cone subsequence, there are
elements $g_\alpha\in G_\xi$ such that
$g_\alpha\cdot[W]\to[C(Y)]$ in $\operatorname{Hilb}_h^T$.
\end{lemma}

\begin{proof}
Choose $D$ large enough so that the coordinate ring $R(C(Y))$ is generated by
$E_D(C(Y))=\bigoplus_{0<d<D}R_d(C(Y))$. Let
$d_1<\cdots<d_{k_0}<D$ be the corresponding charges. By
\cite[Proposition~3.14]{DSII}, the metric filtration
$I_k=\{f\in A:\nu(f)\geq d_k\}$ admits decompositions
$I_k=I_{k+1}\oplus J_k$, where each $J_k$ has an adapted sequence of bases of
charge $d_k$. Set $P:=\bigoplus_{0<k\leq k_0}J_k$.

The adapted bases of $P$ define holomorphic maps
$F_\alpha:B_\alpha\to\mathbb C^N$. After passing to a subsequence, and after
absorbing the compact $K_\xi$ ambiguity, Donaldson--Sun obtain
$F_\alpha\to\Phi:C(Y)\to\mathbb C^N$ locally uniformly, where $\Phi$ is the
homogeneous embedding by an orthonormal basis of $E_D(C(Y))$. The same adapted
bases also embed $W=\operatorname{Spec}\operatorname{gr}_\nu A$ into
$\mathbb C^N$; denote the image by $W_\alpha$. Equivalently,
$W_\alpha=g_\alpha W$ for some $g_\alpha\in G_\xi$.

Let $S_k$ be the space of weighted homogeneous polynomials of weighted degree
$d_k$ on $\mathbb C^N$. Let $V_k\subset S_k$ be the kernel of the restriction
map $S_k\to R_{d_k}(C(Y))$, and fix a complement $S_k=V_k\oplus Q_k$. For
large $\alpha$, let $T_{k,\alpha}:=F_\alpha^*Q_k\subset A$. By
\cite[Lemma~3.15]{DSII}, $I_k=T_{k,\alpha}\oplus I_{k+1}$.

We now recall the key step in \cite[Proposition~3.16]{DSII}. Take
$f\in V_k$. Since $F_\alpha^*f\in I_k$, the splitting
$I_k=T_{k,\alpha}\oplus I_{k+1}$ gives a unique decomposition
$F_\alpha^*f=g_\alpha'+F_\alpha^*h_\alpha$, with
$g_\alpha'\in I_{k+1}$ and $h_\alpha\in Q_k$. The proposition shows that
$h_\alpha\to0$ in the finite-dimensional space $Q_k$. Indeed, if not, after
normalizing one can assume $\|h_\alpha\|=1$ and pass to a limit
$h_\infty\neq0$. Since $F_\alpha\to\Phi$ and $f$ vanishes on $\Phi(C(Y))$, the
limits satisfy $g_\infty+h_\infty=0$. Thus $g_\infty$ is a nonzero homogeneous
function of charge $d_k$. But $g_\alpha'\in I_{k+1}$ has metric order
strictly larger than $d_k$, so its nonzero tangent-cone limit cannot have charge
$d_k$. This contradiction proves $h_\alpha\to0$.

Therefore, for every relation $f\in V_k$ on $C(Y)$, there are relations
$f_\alpha:=f-h_\alpha$ cutting out $W_\alpha$ in weighted degree $d_k$, with
$f_\alpha\to f$. Taking a basis of $V_k$ for all $k$ up to the bounded degree
which generates the multigraded Hilbert ideal, the defining homogeneous pieces
of the ideals of $W_\alpha$ converge to those of $C(Y)$. Hence
$[W_\alpha]\to[C(Y)]$ in the multigraded Hilbert scheme. Since
$W_\alpha=g_\alpha W$, this is the desired convergence
$g_\alpha\cdot[W]\to[C(Y)]$.
\end{proof}

\begin{corollary}The affine variety $W$ is normal.\end{corollary}

\begin{proof}
By the previous lemma, after replacing $W$ by embedded copies $W_i$ in its
$G_\xi$-orbit, we have $[W_i]\to[C(Y)]$ in $\operatorname{Hilb}^T$.
The universal family over this Hilbert scheme is flat. Since $C(Y)$ is normal and normality is open in flat families, all fibers over a sufficiently small neighborhood of $[C(Y)]$ are normal. Thus $W_i$ is normal for $i$ large. Since $W_i$ is obtained from $W$ by an element of $G_\xi$, $W$ is normal.
\end{proof}

Let $(X_\infty,\Delta_{X_\infty})$ be the pair obtained from the weak conical KE equation. By the proof of the structure theorem, if $\mathcal N$ is the uniform denominator for the approximating family $\mathcal F$, then $\mathcal N\Delta_{X_\infty}$ is an integral Weil divisor and $\mathcal N(K_{X_\infty}+\Delta_{X_\infty})$ is Cartier. 

Since $X_\infty$ is normal and Noetherian, every Weil divisor has a unique finite decomposition as a $\mathbb Z$-linear combination of prime Weil divisors, i.e. irreducible codimension-one subvarieties. Thus, near $p$, after discarding components not passing through $p$, we may write $\Delta_{X_\infty}=\sum_a(1-\beta_a)D_a$, where each $D_a$ corresponds to a height-one prime ideal $\mathfrak p_a\subset A:=\mathcal O_{X_\infty,p}$.

If $D_a$ happens to be Cartier at $p$, say $\mathfrak p_a=(f_a)$, then the first degeneration is easy to describe: $\operatorname{in}_\nu(\mathfrak p_a)=(\operatorname{in}_\nu f_a)$, and $D_{a,W}$ is the codimension-one part of $\operatorname{div}_W(\operatorname{in}_\nu f_a)$. But in general $D_a$ is only a prime Weil divisor. At the generic point of $D_a$, it is Cartier because $X_\infty$ is normal, equivalently $\mathcal O_{X_\infty,\eta_a}$ is a DVR. However, a local equation at the generic point need not extend to an element of $\mathcal O_{X_\infty,p}$. Thus the correct local datum at $p$ is the height-one prime ideal $\mathfrak p_a$, not a single defining function. 

The components $D_a$ need not be $\mathbb Q$-Cartier. Thus the first
degeneration must be defined from the height-one prime ideal
$\mathfrak p_a$, rather than from a chosen local defining equation.
Moreover, degeneration need not preserve primeness: a prime divisor may split
into several codimension-one components or acquire multiplicity in the
associated graded ring. The resulting object is therefore naturally a Weil
cycle. We now make this construction precise.

\begin{definition} [Divisors in $W$]
For each $j$ with $p\in D_j$, let $\mathfrak p_j\subset A$ be the height-one prime ideal defining $D_j$ at $p$. Define the initial ideal $\operatorname{in}_\nu(\mathfrak p_j)
        =
        \bigoplus_{d\in S}
        \frac{\mathfrak p_j\cap I_d+I_{>d}}{I_{>d}}
        \subset R_p.
$
Then define $D_{j,W}
        :=
        \left[
        \operatorname{Spec}
        \left(R_p/\operatorname{in}_\nu(\mathfrak p_j)\right)
        \right]_{n-1}.
$
Finally, set
\[
\Delta_W:=\sum_{j:p\in D_j}(1-\beta_j)D_{j,W}.
\]
\end{definition}

Choose an integer $N_\Delta>0$ such that
\[
q_j:=N_\Delta(1-\beta_j)\in\mathbb Z
\]
for every component under consideration, and set
\[
\Gamma_W:=N_\Delta\Delta_W
=
\sum_{j:p\in D_j}q_jD_{j,W}.
\]

Let $d$ be the degree of the projective closure $\overline{\Gamma_W}\subset
\mathbb P^N$. Let $\operatorname{Chow}_{n-1,d}^T$ denote the projective Chow variety parametrizing $T$-invariant effective $(n-1)$-dimensional cycles of
degree $d$ in $\mathbb P^N$.

We retain only the pure codimension-one cycle because the log boundary is a
$\mathbb{Q}$-Weil divisor. This forgets embedded components, nilpotent structure, and
higher-codimension scheme structure, but it retains the support and
generic multiplicities needed for the boundary coefficients.

\begin{definition}[Hilbert--Chow incidence space]
Let
\[
    \mathcal P
    :=
    \left\{
    ([Z],[\Gamma])\in
    \operatorname{Hilb}_h^T\times \operatorname{Chow}_{n-1,d}^T
    :
    \operatorname{Supp}\Gamma\subset \overline Z
    \right\}.
\]
Here $\overline Z\subset\mathbb P^N$ denotes the projective closure of the
affine scheme $Z\subset\mathbb C^N$.
\end{definition}

By the positivity argument above, $\operatorname{Hilb}_h^T$ is projective.
The projective Chow variety of effective cycles of fixed dimension and degree is
standard; see \cite[Chapter~I]{KollarRationalCurves}. Hence the closed
$T$-fixed locus $\operatorname{Chow}_{n-1,d}^T$ is projective as well.
and the incidence condition
$\operatorname{Supp}\Gamma\subset\overline Z$
is closed in families. Thus $\mathcal{P}$ is projective. Moreover, the action of $G_\xi$ on $\mathbb C^N$ induces
actions on both factors and preserves the incidence condition, so
$\mathcal P$ is a projective $G_\xi$-scheme.

Let $x_W:=([W],[\Gamma_W])\in \mathcal P$, and $g_\alpha W\to C(Y)$ be the convergence from Lemma~\ref{lem:ambient-dsii-degeneration} in $\operatorname{Hilb}_h^T$. Since $g_\alpha\Gamma_W$ is a sequence of effective $T$-invariant codimension-one cycles of fixed degree $d$, projectivity of
$\operatorname{Chow}_{n-1,d}^T$ implies that, after passing to a subsequence, $g_\alpha\Gamma_W\longrightarrow \Gamma_C^{\mathrm{alg}}$
in $\operatorname{Chow}_{n-1,d}^T$. We define the central boundary cycle using this. 

\begin{definition}[Algebraic central boundary cycle]
For such a subsequential Chow limit, define $\Gamma_C^{\mathrm{alg}}
    :=
    \lim_{\alpha\to\infty}g_\alpha\Gamma_W,
    \qquad
    \Delta_C^{\mathrm{alg}}
    :=
    \frac{1}{N_\Delta}\Gamma_C^{\mathrm{alg}}.
$
\end{definition}

Note that, a priori, $\Gamma_C^{\mathrm{alg}}$ may depend on the chosen
Chow-convergent subsequence. At this stage it is only a candidate central
boundary cycle. The point of the comparison theorem below is precisely to show
that every such limit agrees with $N_\Delta\Delta_C^{\mathrm{met}}$.

By Lemma~\ref{lem:ambient-dsii-degeneration}, $g_\alpha W\to C(Y)$ in
$\operatorname{Hilb}_h^T$, and by construction
$g_\alpha\Gamma_W\to \Gamma_C^{\mathrm{alg}}$ in
$\operatorname{Chow}_{n-1,d}^T$. Since
$\operatorname{Supp}(g_\alpha\Gamma_W)\subset \overline{g_\alpha W}$ for every
$\alpha$, closedness of the incidence condition gives
$x_C^{\mathrm{alg}}:=([C(Y)],[\Gamma_C^{\mathrm{alg}}])\in\mathcal P$.
Moreover $x_C^{\mathrm{alg}}\in\overline{G_\xi\cdot x_W}$.

The following warning is important. Hilbert--Chow compactness alone does not
identify $\Gamma_C^{\mathrm{alg}}$ with the metric blow-up of the boundary. If
$U\subset W$ is an arbitrary cycle and $g_\alpha(W,U)\to(C(Y),U')$, there is no
reason for $U'$ to be a metric blow-up of anything on $X_\infty$. The
identification is special to cycles which come from filtered ideals in
$A=\mathcal O_{X_\infty,p}$, because the adapted-basis maps $g_\alpha$ are the
change-of-adapted-basis maps. In that case the equations defining
$g_\alpha U$ are the leading terms of the rescaled analytic equations on
$X_\infty$, so the Chow limit and the metric limit are cut out by the same
limiting equations.

Write
\[
\Delta_{X_\infty}=\sum_j a_jD_j
\]
near $p$, and recall the integer $N_\Delta$ fixed above. Thus
$q_j:=N_\Delta a_j\in\mathbb Z$, and we set
\[
\Gamma_{X_\infty}:=N_\Delta\Delta_{X_\infty}
=
\sum_j q_jD_j.
\]

For a component $D_j$, let $\mathfrak p_j\subset A$ be its height-one ideal.
Define $\operatorname{in}_\nu(\mathfrak p_j)\subset\operatorname{gr}_\nu A$ to
be the homogeneous ideal generated by all $\operatorname{in}_\nu(f)$ with
$f\in\mathfrak p_j$. The first degeneration of $D_j$ is
$D_{j,W}:=[\operatorname{Spec}(\operatorname{gr}_\nu A/
\operatorname{in}_\nu(\mathfrak p_j))]_{n-1}$, and
Thus the previously defined cycle satisfies
\[
\Gamma_W=\sum_j q_jD_{j,W}.
\]

\begin{lemma}[Weil boundary components and the two-step DSII degeneration]
\label{lem:weil-boundary-two-step}
Let the adapted-basis maps $g_\alpha$ define the second degeneration
$g_\alpha W\to C(Y)$ attached to the chosen tangent-cone subsequence.
After passing to a Chow-convergent subsequence,
$g_\alpha\Gamma_W\to\Gamma_C^{\mathrm{alg}}$. Then
$\Gamma_C^{\mathrm{alg}}$ is the codimension-one cycle obtained by taking the
metric blow-up limits of the divisor components $D_j$ along the same
tangent-cone subsequence and weighting them by $q_j$.

Equivalently, if $E$ is an irreducible component of
$\operatorname{Supp}\Gamma_C^{\mathrm{alg}}$ and $q\in E$ is a generic smooth
point away from all other components and higher-codimension strata, then
$\operatorname{coeff}_E(\Gamma_C^{\mathrm{alg}})=\sum_j q_jm_{j,E}$, where
$m_{j,E}$ is the generic multiplicity with which the metric blow-up of $D_j$
contains $E$.
\end{lemma}

\begin{proof}
Fix $j$. The quotient filtration on $A/\mathfrak p_j$ gives
$\operatorname{gr}_\nu(A/\mathfrak p_j)\cong
\operatorname{gr}_\nu A/\operatorname{in}_\nu(\mathfrak p_j)$. Thus $D_{j,W}$
is the pure codimension-one part of the Rees degeneration of the germ
$(D_j,p)$. This formulation uses the ideal $\mathfrak p_j$, so no local
principal equation for $D_j$ is required.

Choose homogeneous generators $\bar f_{j\ell}$ of
$\operatorname{in}_\nu(\mathfrak p_j)$ in the degrees needed for the fixed
Hilbert--Chow embedding, and lift them to elements
$f_{j\ell}\in\mathfrak p_j$ with
$\operatorname{in}_\nu(f_{j\ell})=\bar f_{j\ell}$. Write
$e_{j\ell}:=\nu(f_{j\ell})$. For each $\alpha$, the adapted-basis property gives
weighted homogeneous polynomials $q_{j\ell,\alpha}$ such that
$f_{j\ell}=F_\alpha^*q_{j\ell,\alpha}+r_{j\ell,\alpha}$, with
$\nu(r_{j\ell,\alpha})>e_{j\ell}$.

The equations $q_{j\ell,\alpha}=0$ cut out $g_\alpha D_{j,W}$ inside
$g_\alpha W$, up to embedded and higher-codimension components. On the other
hand, after rescaling the original analytic germ $D_j$, the equations
$f_{j\ell}=0$ have the same leading equations $q_{j\ell,\alpha}=0$, since the
errors $r_{j\ell,\alpha}$ have strictly larger metric order. Passing to the
tangent-cone limit, $q_{j\ell,\alpha}$ converges to a homogeneous function
$q_{j\ell,\infty}$ on $C(Y)$. Therefore both the Chow limit of
$g_\alpha D_{j,W}$ and the metric blow-up of $D_j$ are cut out, in pure
codimension-one, by the same limiting ideal
$(q_{j\ell,\infty})_\ell$ on $C(Y)$.

Taking the weighted sum over $j$ gives the statement for
$\Gamma_W=\sum_j q_jD_{j,W}$. The coefficient formula follows by evaluating at a
generic smooth point of an irreducible component $E$: at such a point the
codimension-one cycle multiplicity is exactly the transverse order of vanishing
of the limiting ideal, and the total coefficient is the sum of the contributions
$q_jm_{j,E}$.
\end{proof}

\begin{theorem}[Algebraic and metric boundaries agree]
\label{thm:alg-met-boundary-agree}
We have $\Gamma_C^{\mathrm{alg}}=N_\Delta\Delta_C^{\mathrm{met}}$ as
codimension-one Chow cycles. Equivalently,
$\Delta_C^{\mathrm{alg}}=\Delta_C^{\mathrm{met}}$ as $\mathbb Q$-Weil divisors
on $C(Y)$.
\end{theorem}

\begin{proof}
Let $\Gamma_C^{\mathrm{br}}$ denote the codimension-one cycle-theoretic blow-up
limit of $\Gamma_{X_\infty}=N_\Delta\Delta_{X_\infty}$ along the chosen
tangent-cone subsequence. Thus $\Gamma_C^{\mathrm{br}}$ is obtained by taking
the local analytic limits of the rescaled divisor cycles and retaining their
pure codimension-one part.

The adapted-basis comparison gives
$\Gamma_C^{\mathrm{alg}}=\Gamma_C^{\mathrm{br}}$. Indeed, for each prime
component $D_j$ with ideal $\mathfrak p_j\subset A=\mathcal O_{X_\infty,p}$, the
first degeneration is
$\operatorname{Spec}(\operatorname{gr}_\nu A/\operatorname{in}_\nu\mathfrak p_j)$.
The adapted-basis maps $g_\alpha$ are the change-of-adapted-basis maps, so the equations
cutting out $g_\alpha D_{j,W}$ are exactly the leading adapted equations for
the rescaled analytic germ $D_j$. Hence the Chow limit of $g_\alpha D_{j,W}$ is
the codimension-one blow-up limit of $D_j$. Taking the weighted sum over $j$
gives $\Gamma_C^{\mathrm{alg}}=\Gamma_C^{\mathrm{br}}$.

It remains to identify $\Gamma_C^{\mathrm{br}}$ with
$N_\Delta\Delta_C^{\mathrm{met}}$. By
Lemma~\ref{lem:boundary-support-convergence}, the supports of the rescaled
$\Delta_{X_\infty}$ converge locally in the Kuratowski sense to
$\Delta_C^{\mathrm{met}}$. Therefore
$\operatorname{Supp}\Gamma_C^{\mathrm{br}}=\operatorname{Supp}\Delta_C^{\mathrm{met}}$.

Now compare coefficients. Let $E$ be an irreducible component of this common
support, and choose a generic point
$q\in E^{reg}\cap C(Y)^{reg}\cap(S_2(C(Y))\setminus S_4(C(Y)))$
away from all other components. Choose local coordinates
$(u,v_1,\dots,v_{n-1})$ near $q$ with $E=\{u=0\}$, and choose a small transverse
disk $T=\{v=v_0,\ |u|<\epsilon\}$ meeting $E$ once and no other component.

By the structure theorem, the metric boundary coefficient $b_E$ is characterized
by
$\omega_C^n=e^G|u|^{-2b_E}dV$ near $q$, with $G$ bounded. Equivalently, in the
approximating rescaled pairs, if the branches of
$\Delta_{X_\infty}$ meeting the corresponding disks $T_\alpha$ converge to
$E$, and if these branches come from components $D_{j_\ell}$ with coefficients
$a_{j_\ell}$ and transverse multiplicities $m_\ell$, then the local
Monge--Amp\`ere equation gives
$b_E=\sum_\ell a_{j_\ell}m_\ell$.
This is the same coefficient calculation as in the tubular trapping/holonomy
lemma: the exponent of the limiting factor $|u|^{-2b_E}$ is the total angle
defect of the branches collapsing to $E$.

On the other hand, the coefficient of $\Gamma_C^{\mathrm{br}}$ along $E$ is
computed by intersection with the same generic transverse disk $T$. Since
$\Gamma_{X_\infty}=\sum_j q_jD_j$ with $q_j=N_\Delta a_j$, the intersections
with $T_\alpha$ have total multiplicity
$\sum_\ell q_{j_\ell}m_\ell$. Passing to the cycle limit gives
$\operatorname{coeff}_E(\Gamma_C^{\mathrm{br}})
=\sum_\ell q_{j_\ell}m_\ell
=N_\Delta\sum_\ell a_{j_\ell}m_\ell
=N_\Delta b_E$.

Thus every component $E$ has the same coefficient in
$\Gamma_C^{\mathrm{br}}$ and $N_\Delta\Delta_C^{\mathrm{met}}$. Hence
$\Gamma_C^{\mathrm{br}}=N_\Delta\Delta_C^{\mathrm{met}}$. Since
$\Gamma_C^{\mathrm{alg}}=\Gamma_C^{\mathrm{br}}$, the theorem follows.
\end{proof}
For the local slice argument below we work with the reduced projective orbit
closure $P_W:=\overline{G_\xi\cdot x_W}_{\mathrm{red}}$; no nonreduced scheme
structure is used in this step.
We now record the Donaldson--Sun equivariant projective-slice argument in the
form needed for uniqueness of the tangent cone; see the proof of
\cite[Theorem~1.3]{DSII}.
\begin{lemma}[Local orbit rigidity from an equivariant projective slice]
Let $G$ be a complex reductive group acting on a projective variety $P$, and let
$Z\subset P$ be a connected compact subset. Suppose
$Z\subset\overline{G\cdot x}$ for some fixed point $x\in P$. Let $o\in Z$ such
that
\begin{itemize}
    \item the orbit $G\cdot o$ is closed;
    \item $H=\operatorname{Stab}_G(o)$ is reductive, and
    $\dim H$ is minimal among points in $Z$;
    \item an $H$-invariant equivariant projective slice $S$ exists at $o$, as in
    the proof of \cite[Theorem~1.3]{DSII};
    \item if $z\in Z$ is sufficiently close to $o$ and $g\cdot z=s\in S$, then
    $o\in\overline{H\cdot s}$.
\end{itemize}
Then every $z\in Z$ sufficiently close to $o$ lies in $G\cdot o$.
\end{lemma}

\begin{proof}
Move $z$ into the slice, say $s=g\cdot z\in S$. By the slice-closure
assumption, $o\in\overline{H\cdot s}\subset\overline{G\cdot z}$. If
$z\notin G\cdot o$, then $G\cdot o$ is a boundary orbit in
$\overline{G\cdot z}$, so $\dim G\cdot o<\dim G\cdot z$. Equivalently,
$\dim\operatorname{Stab}_G(o)>\dim\operatorname{Stab}_G(z)$, contradicting the
choice of $o$ with minimal $G$-stabilizer dimension on $Z$. Therefore
$z\in G\cdot o$.
\end{proof}

\begin{theorem}[Uniqueness of the log tangent cone]
For the fixed point $p\in X_\infty$, the log metric tangent cone
$(C(Y),\Delta_C,\omega_C)$ is unique.
\end{theorem}

\begin{proof}
Let $\mathcal P$ be the Hilbert--Chow incidence space, and let
$\mathcal Z_p \subset \mathcal P$ be the compact connected set of log tangent cone
points
$x_C=([C(Y)],[N_\Delta\Delta_C])$.
Let
$x_W=([W],[N_\Delta\Delta_W])$
and let $G=G_\xi$.

By the Hilbert--Chow degeneration constructed above, every
$x_C\in \mathcal Z_p$ lies in $\overline{G\cdot x_W}$. Choose
$x_C\in \mathcal Z_p$ with minimal stabilizer dimension. By log polystability of
the conical Ricci-flat cone, the orbit $G\cdot x_C$ is closed. By log
Matsushima,
$H=\operatorname{Stab}_G(x_C)=\operatorname{Aut}(C(Y),\Delta_C,\xi)$
is reductive, so Luna's slice theorem applies at $x_C$.

The Donaldson--Sun projective-slice argument \cite[Theorem~1.3]{DSII} then gives local orbit rigidity: all log tangent
cone points sufficiently close to $x_C$ lie in $G\cdot x_C$. Applying the same
argument on the compact connected set $\mathcal Z_p$ shows that all log tangent
cone points lie in a single $G$-orbit. Thus any two tangent cones are
isomorphic as log affine cones with Reeb field.

Finally, by uniqueness of conical Ricci-flat Kähler cone metrics on a fixed log
affine cone $(C,\Delta,\xi)$, the corresponding metric cones are isometric.
Hence the tangent cone at $p$ is unique.
\end{proof}

\section{Special Test Configurations and Stable Degeneration}

The goal of this section is to upgrade the orbit-closure relation
$x_C\in\overline{G_\xi\cdot x_W}$ to a one-parameter degeneration of log Fano
cones by test configurations, and then to compare this construction with the
stable-degeneration framework of Li--Liu--Xu and Li--Wang--Xu.

We first construct the candidate one-parameter degeneration. The following
lemma produces the required one-parameter subgroup. The closedness of
$G_\xi\cdot x_C$ in $\mathcal P$ should be viewed as the finite-dimensional
GIT form of log polystability for the embedded log Fano cone.

\begin{corollary}[One-parameter orbit degeneration]
Assume that the orbit $G_\xi\cdot x_C$ is closed in $\mathcal P$ and that
$x_C\in\overline{G_\xi\cdot x_W}$. Then there exist $g\in G_\xi$ and a
one-parameter subgroup $\lambda:\mathbb C^*\to G_\xi$ such that
$\lim_{t\to0}\lambda(t)\cdot g x_W=x_C$.

Consequently, pulling back the universal Hilbert--Chow incidence family along
the morphism $\mathbb A^1\to \mathcal{P}$ determined by this one-parameter limit
gives a flat one-parameter degeneration of $W$ to $C(Y)$, together with a
relative codimension-one cycle degenerating $N_\Delta\Delta_W$ to
$N_\Delta\Delta_C$.
\end{corollary}

\begin{proof}
By assumption, $G_\xi\cdot x_C$ is the closed orbit contained in
$\overline{G_\xi\cdot x_W}$. The Hilbert--Mumford--Birkes orbit-closure theorem
therefore gives $g\in G_\xi$ and a one-parameter subgroup
$\lambda:\mathbb C^*\to G_\xi$ such that
$\lim_{t\to0}\lambda(t)\cdot g x_W\in G_\xi\cdot x_C$. Composing $g$ with a
suitable element of $G_\xi$, we may arrange that the limit is exactly $x_C$.

The map $t\mapsto \lambda(t)\cdot g x_W$ for $t\neq0$ extends to a morphism
$\mathbb A^1\to \mathcal P$ by the existence of the above limit. Pulling back the
universal Hilbert family gives a flat family of affine schemes. Pulling back
the universal Chow cycle gives the corresponding relative codimension-one
cycle. The general fiber is $(gW,gN_\Delta\Delta_W)$ and the central fiber is
$(C(Y),N_\Delta\Delta_C)$.
\end{proof}

There is a natural way to refine the Hilbert--Chow construction so that the
boundary itself is retained as a divisor in the family. We briefly explain this
construction, and isolate the additional compatibility which is needed for the
stable-degeneration interpretation.

Let $\mathcal H$ be the Hilbert locus containing the projective closures of
$W$, $C(Y)$, and their $G_\xi$-translates, and let
$\mathfrak X\to\mathcal H$ be the universal family. After fixing
$N_\Delta$ with $\Gamma_Z:=N_\Delta\Delta_Z$ integral for all pairs under consideration, Koll\'ar's space $\operatorname{KDiv}_{d}(\mathfrak X/\mathcal H)$
parametrizes K-flat relative Mumford divisors of the prescribed degree
\cite{KollarFamiliesDivisors}. Over the normal-fiber locus this space is proper
over $\mathcal H$.

Thus the Hilbert--Chow degeneration constructed above may be lifted, after
restricting to the corresponding curve in the Hilbert space, to a family of
Mumford divisors. Indeed, away from the central point the family is obtained by
translating the fixed pair $(W,\Gamma_W)$, while the central Chow cycle is
$\Gamma_C=N_\Delta\Delta_C$ by
Theorem~\ref{thm:alg-met-boundary-agree}. Since the fibers $W$ and $C(Y)$ are
normal, the resulting KDiv limit recovers the same Weil divisor on the central
fiber.

\begin{remark} This does not by itself imply that the degeneration is log
$\mathbb Q$-Gorenstein. The missing information is carried by the reflexive
log-pluricanonical sheaves $\omega_Z^{[m]}(m\Delta_Z)$. Koll\'ar's hull theory gives a locally closed decomposition of the base on which the corresponding reflexive hulls are flat and commute with base change;
moreover there are locally closed loci on which these hulls are invertible
\cite[Cor.~3.30, Prop.~3.31]{KollarFamiliesGeneralType}. 

However, it is not automatic that the Rees degeneration and the
one-parameter degeneration arising from the metric filtration factor through
one of these loci. One must still control specialization of the
log-canonical reflexive sheaf.
\end{remark}

We therefore make the following compatibility assumption.

\begin{assumption}[Associated-graded log-canonical compatibility]
\label{ass}
For some sufficiently divisible $m$, the filtrations defining the two
Donaldson--Sun degenerations extend to the log-pluricanonical modules and
satisfy
$
\left(
\operatorname{gr}_{\nu}
(\omega_{X_\infty}^{[m]}(m\Delta_\infty)
\right)^{}
\cong
\omega_W^{[m]}(m\Delta_W),
$
and
$
\left(
\operatorname{gr}_{\eta}
\omega_W^{[m]}(m\Delta_W)
\right)^{}
\cong
\omega_{C(Y)}^{[m]}(m\Delta_C).
$
\end{assumption}

Under Assumption~\ref{ass}, $(W,\Delta_W)$ is log
$\mathbb Q$-Gorenstein and the second Donaldson--Sun degeneration is a relative
log $\mathbb Q$-Gorenstein degeneration of pairs.

\begin{lemma}[Relative log $\mathbb Q$-Gorensteinness]
\label{lem:relative-log-qg}
Let $(\mathcal W,\mathcal D)\to\mathbb A^1$ be the one-parameter family obtained from the one-parameter subgroup degeneration. Then $\mathcal W$ is normal and $K_{\mathcal W/\mathbb A^1}+\mathcal D$ is $\mathbb Q$-Cartier.
\end{lemma}

\begin{proof}
Flatness follows because $\mathcal W\to\mathbb A^1$ is pulled back from the universal Hilbert family. The nonzero fibers are isomorphic to the normal variety $W$, and the central fiber is the normal variety $C(Y)$. By Stacks Project, More on Morphisms, Definition 37.20.1 and Lemma 37.20.2, the morphism $\mathcal W\to\mathbb A^1$ is normal: it is flat and all geometric fibers are normal. Since $\mathbb A^1$ is normal, EGA IV, Proposition 6.8.1 implies that $\mathcal W$ is normal. Since the base $\mathbb A^1$ is normal and the morphism is flat with geometrically normal fibers, the total space $\mathcal W$ is normal.

It remains to show that $m(K_{\mathcal W/\mathbb A^1}+\mathcal D)$ is Cartier. This is local on $\mathcal W$. Over $\mathbb A^*$, the family is equivariantly isomorphic to $(W,\Delta_W)\times\mathbb A^*$. Hence $\omega_{\mathcal W/\mathbb A^1}^{[m]}(m\mathcal D)$ is locally free over $\mathbb A^*$.

Let $V$ be an open subset of $W$ on which $\omega_W^{[m]}(m\Delta_W)$ is generated by a nowhere-vanishing section $\tau_V$. Transport $\tau_V$ along the product trivialization over $\mathbb A^*$, and view the result as a rational section $\tau_V^*$ of $\omega_{\mathcal W/\mathbb A^1}^{[m]}(m\mathcal D)$ on the closure of $V\times\mathbb A^*$ in $\mathcal W$. Since $\tau_V^*$ is nowhere vanishing on every noncentral fiber, its divisor has no horizontal component. Here “nowhere-vanishing” means as a section of the log pluricanonical sheaf; viewed as an ordinary pluricanonical form, $\tau_V$ may have poles along
$\Delta_W$. The only possible vertical codimension-one component is the central fiber $\mathcal W_0$. Therefore $\operatorname{div}(\tau_V^*)=a_V\mathcal W_0$ for some integer $a_V$.

Since $\mathcal W_0=\operatorname{div}(t)$ is principal, the section $t^{-a_V}\tau_V^*$ has zero divisor. On a normal variety, a rational section of a rank-one reflexive sheaf with zero divisor gives a local trivialization. Thus $\omega_{\mathcal W/\mathbb A^1}^{[m]}(m\mathcal D)$ is locally free near the closure of $V\times\mathbb A^*$. Such opens cover a neighborhood of the central fiber by the one-parameter degeneration. Therefore $m(K_{\mathcal W/\mathbb A^1}+\mathcal D)$ is Cartier.
\end{proof}

\begin{theorem}[Special test configuration]
The family $(\mathcal W,\mathcal D,\xi;\eta)\to\mathbb A^1$ is a
$T$-equivariant special test configuration of the log Fano cone
$(W,\Delta_W,\xi)$ with central fiber $(C(Y),\Delta_C,\xi)$.

Moreover, $(W,\Delta_W,\xi)$ is K-semistable, and
$(C(Y),\Delta_C,\xi)$ is its K-polystable, hence stable, degeneration in the
sense of Li--Wang--Xu.
\end{theorem}

\begin{proof}
We first verify the definition of a special test configuration of a log Fano
cone.

The family $\mathcal W\to\mathbb A^1$ is affine and flat, since it is obtained
from the one-parameter orbit closure in the affine Hilbert scheme, together with
the pullback of the universal family. Its general fiber is isomorphic to $W$,
and its central fiber is $C(Y)$. The relative boundary $\mathcal D$ is the
cycle-theoretic closure of the moving boundary over $\mathbb A^*$; it contains
no fiber component, its general fiber is $\Delta_W$, and by the boundary
comparison theorem its central fiber is $\Delta_C$.

The one-parameter subgroup $\lambda:\mathbb C^*\to G_\xi$ gives a
$\mathbb C^*$-action on $(\mathcal W,\mathcal D)$ covering the standard action
on $\mathbb A^1$. With the sign convention of Li--Wang--Xu
\cite[Definition~2.23]{LiWangXu}, the generating
vector field is denoted $\eta$ and satisfies $\pi_*\eta=-t\partial_t$. Since
$\lambda\subset G_\xi$ preserves the $T$-weight decomposition, this
$\mathbb C^*$-action commutes with the fiberwise $T$-action. Hence the Reeb
field $\xi$ extends over the family.

By the normality lemma, $\mathcal W$ is normal. By the relative log
$\mathbb Q$-Gorenstein lemma, $K_{\mathcal W/\mathbb A^1}+\mathcal D$ is
$\mathbb Q$-Cartier. Finally, the central fiber $(C(Y),\Delta_C)$ is klt by the
cone-side structure theorem. Therefore $(\mathcal W,\mathcal D,\xi;\eta)$ is a special test configuration
in the sense of Li--Wang--Xu \cite[Definition~2.23]{LiWangXu}.

We next prove the stability assertions. The central fiber
$(C(Y),\Delta_C,\xi)$ carries the weak conical Ricci-flat K\"ahler cone metric
constructed above. By the local Yau--Tian--Donaldson theorem for log Fano cone
singularities, equivalently the Ricci-flat cone/Ding-polystability theorem for
log Fano cones, $(C(Y),\Delta_C,\xi)$ is K-polystable. In the boundary-free case
this is \cite[Theorem~1.1]{CollinsSzekelyhidi} or
\cite[Corollary~A.4]{LiWangXu}; in the log Fano cone setting we use
\cite[Theorem~2.9]{ChiLiWeightedSolitons}.

Since K-polystability implies K-semistability, the central fiber is
K-semistable. We may therefore apply \cite[Proposition~5.5]{LLXGuidedTour} to
the special degeneration
$(W,\Delta_W,\xi)\rightsquigarrow(C(Y),\Delta_C,\xi)$. This gives that
$(W,\Delta_W,\xi)$ is K-semistable.

Now Li--Wang--Xu's uniqueness theorem for K-polystable degenerations of log
Fano cones applies: a K-semistable log Fano cone admits a special degeneration
to a K-polystable log Fano cone, and that K-polystable central fiber is unique
up to isomorphism. Hence the K-polystable cone
$(C(Y),\Delta_C,\xi)$ obtained above is the stable degeneration of
$(W,\Delta_W,\xi)$.
\end{proof}

\section{Further questions}

We conclude by recording three questions suggested by the points at which the
hypotheses above enter the argument.

First, the comparison with stable degeneration in the final section still
depends on Assumption~\ref{ass}. It would be interesting to know
whether this compatibility is automatic for the metric filtration, or whether
it follows from natural singularity-theoretic assumptions on the pairs
involved. 

A second question concerns cone angles approaching $2\pi$.  A basic input in
the proof of Theorem~3.3 is the uniform density gap coming from the fixed
finite set of cone angles; this allows the divisorial conical locus to be
separated from the higher-codimension metric singularities. This mechanism
degenerates when some $\beta_i\to1$.  In the setting of a single smooth
divisor $D_i\in|-\lambda K_{X_i}|$, Chen--Donaldson--Sun treat this regime in
\cite{CDSIII}. Their second approach
uses the stronger smooth approximation for which the Ricci lower bound
approaches the Einstein lower bound, followed by Ricci flow and the
Tian-Wang pseudolocality theorem. It would be interesting to understand to what
extent these arguments extend to general log pairs, where the boundary need
not be a single smooth divisor, and in particular when several boundary
coefficients tend to zero simultaneously.

Finally, one can ask how essential the polarization hypothesis is for the
algebraic structure of the limit.  Liu--Sz\'ekelyhidi \cite[Theorem~1.1]{SLI} show that a
non-collapsed Gromov--Hausdorff limit of polarized K\"ahler manifolds with a
uniform Ricci lower bound is a normal projective variety. Suppose instead that each $X_i$ is smooth and
projective and that
$\operatorname{Ric}(\omega_i)\geq-C\omega_i$,
$\operatorname{diam}(X_i,\omega_i)\leq D$, and
$\operatorname{Vol}(X_i,\omega_i)\geq v>0$, but that no uniformly controlled
polarization compatible with the metrics is assumed.  Can a genuinely
singular Gromov--Hausdorff limit fail to be homeomorphic to any complex
projective variety?  It would also be interesting to ask whether such a
phenomenon can occur for conical metrics with a fixed cone angle.

\appendix
\section{A lower-Ricci K\"ahler limit with nonunique tangent cones}
\label{app:nonunique-tangent-cones}

The construction is a K\"ahler version of the mechanism in Perelman's positive-Ricci example with nonunique asymptotic cone and its systematic formulation by Colding--Naber \cite{PerelmanNonunique,ColdingNaberTangents}.  In the context of the present paper, it shows that the polarized lower-Ricci framework supplied by Lemma~\ref{lem:polarized-approximation} does not by itself imply the uniqueness conclusion of Theorem~\ref{thm:main-log-ds}.  The K\"ahler--Einstein equation used in the body of the paper is therefore a genuine rigidity input.

\begin{theorem}[Nonuniqueness under a lower Ricci bound]
\label{thm:lower-ricci-nonunique}
There is a sequence of smooth K\"ahler metrics $\omega_j$ on $\mathbf{CP}^2$, all in one fixed integral K\"ahler class, and constants $C,v>0$ such that $\operatorname{Ric}(\omega_j)\geq-C\omega_j$ and $\operatorname{Vol}_{\omega_j}(B_{\omega_j}(p,1))\geq v$ for a fixed point $p\in\mathbf{CP}^2$.  After passing to a subsequence, $(\mathbf{CP}^2,\omega_j,p)$ converges in the pointed Gromov--Hausdorff sense to a compact metric space $(X,d,p)$ whose tangent cone at $p$ is not unique.  In fact the tangent-cone set contains a nontrivial continuous family of pairwise nonisometric K\"ahler cones with the same underlying affine variety $\mathbb C^2$, the same Hopf/Reeb action, and the same volume density.
\end{theorem}

We first record the plan of the construction.  The five steps below correspond to Subsections~\ref{subsec:appendix-riemannian}--\ref{subsec:appendix-compactification}.
\begin{enumerate}
\item We recall the Colding--Naber mechanism.  A volume-preserving deformation of a cone link enters the radial Ricci tensor only quadratically in its speed.  A small radial concavity can therefore pay for a deformation which changes infinitely often but becomes arbitrarily slow at small scales.
\item We construct a path of K\"ahler cones on $\mathbb C^2$ from the Hopf fibration $S^1\to S^3\to\mathbf{CP}^1$.  The transverse K\"ahler metric varies in a fixed cohomology class, so the contact volume is constant.  The cones are pairwise nonisometric after restricting to a sufficiently short parameter interval, even though the affine variety and Reeb field remain fixed.
\item We insert this parameter into the radial variable by means of $T(\log(-s))$, where $s=\log|z|^2$.  The resulting deformation freezes on every blow-up annulus but has several subsequential phases.  The main local obstruction is the radial Ricci error.  We add $K\log(-s)$ to the logarithm of the K\"ahler potential; Lemma~\ref{lem:appendix-slow-modulation} shows that its positive radial contribution dominates the mixed and second-derivative errors.
\item We then realize the singular metric as a Gromov--Hausdorff limit of smooth metrics.  A fixed smooth cap cannot simply be shrunk, since a negative Ricci error would scale like the inverse square of the cap radius.  On radii where the oscillation is frozen at $0$, the metric is radial, and we construct a scale-dependent cap by prescribing the radial Ricci potential and solving one matching equation.
\item Finally we place the local construction in a coordinate ball of $\mathbf{CP}^2$.  A fixed regularized-maximum interpolation joins the common outer part of the local metric to Fubini--Study.  Since this interpolation occurs at a fixed scale, it contributes one fixed Ricci lower-bound constant and leaves the global K\"ahler class unchanged.
\end{enumerate}

\subsection{The Riemannian mechanism}
\label{subsec:appendix-riemannian}

Let $Y^{m-1}$ carry a smooth family $h(\tau)$ with $\operatorname{Ric}(h(\tau))\geq(m-2)h(\tau)$ and pointwise volume form independent of $\tau$.  For $g=dr^2+R(r)^2h(f(r))$, the second fundamental form of the level sets, viewed as an endomorphism, is $A=(R'/R)I+(f'/2)h^{-1}\dot h$.  The volume condition is $\operatorname{tr}_h\dot h=0$, so the part of $A$ which changes the shape of the link is trace-free.  Using $\operatorname{Ric}_{rr}=-\operatorname{tr}(A')-\operatorname{tr}(A^2)$ and differentiating the trace-free identity gives $\operatorname{Ric}_{rr}=-(m-1)R''/R-(f')^2|\dot h|_h^2/4$.  Thus a link deformation costs radial Ricci quadratically in its speed, while a slight concavity $R''<0$ supplies a positive radial reserve.  This is the mechanism in \cite{ColdingNaberTangents}; Perelman's earlier cohomogeneity-one example realizes the same idea by letting the shape of the $S^3$ orbits drift slowly while a separate warping term controls Ricci curvature \cite{PerelmanNonunique}.

The parameter must vary slowly enough to freeze under one blow-up, but it must also fail to converge as $r\to0$.  The choice $f(r)=T(\log|\log r|)$ with $T$ periodic has both properties.  One has $r|f'(r)|=O(|\log r|^{-1})\to0$, and for every fixed $A>1$ the oscillation of $f$ on $A^{-1}r_j<r<Ar_j$ tends to zero as $r_j\to0$.  At the same time $f$ retains the full periodic cluster set along sufficiently separated sequences of scales.  This scale separation is the only ingredient needed later to produce different tangent-cone subsequences.

\subsection{A fixed-volume family of Hopf cone links}
\label{subsec:appendix-hopf}

Let $\pi:\mathbb C^2\setminus\{0\}\to\mathbf{CP}^1$ be the Hopf projection and put $s=\log|z|^2$.  Normalize the Fubini--Study form $\omega_0$ by $\operatorname{Ric}(\omega_0)=2\omega_0$ and $\sqrt{-1}\partial\bar\partial s=\pi^*\omega_0$.  Fix $a\in(0,1)$ and choose a nonconstant $\psi\in C^\infty(\mathbf{CP}^1)$ whose scalar-curvature linearization at $a\omega_0$ is nonzero.  After rescaling $\psi$ we may assume that its $C^4$ norm is as small as needed below.  For $b$ near $a$ and $t\in[0,\varepsilon]$, set $\omega_{b,t}=b\omega_0+t\sqrt{-1}\partial\bar\partial\psi$, $\Phi_{b,t}=\exp(bs+t\psi\circ\pi)$, and $\Omega_{b,t}=\sqrt{-1}\partial\bar\partial\Phi_{b,t}$.

The choice $a<1$ gives a strict horizontal Ricci margin.  Constant rescaling does not change the Ricci form, so $\operatorname{Ric}(b\omega_0)-2b\omega_0=2(1-b)\omega_0$.  After decreasing $\varepsilon$, the size of $\psi$, and the allowed interval for $b-a$, there is $\delta>0$ such that $\omega_{b,t}\geq\delta\omega_0$ and $\operatorname{Ric}(\omega_{b,t})-2\omega_{b,t}\geq4\delta\omega_0$ for all relevant $(b,t)$.

Since $\Phi_{b,t}(\lambda z)=|\lambda|^{2b}\Phi_{b,t}(z)$, the potential is homogeneous.  The corresponding homothetic vector field is a gradient field, and after fixing the cone radius by a constant multiple of $\Phi_{b,t}^{1/2}$ the form $\Omega_{b,t}$ is a K\"ahler cone metric.  Its link is a regular Sasaki metric on the Hopf bundle.  For fixed $b=a$, the Hopf circle and its Reeb field do not depend on $t$, while the transverse K\"ahler form is $\omega_{a,t}$.

We next describe the Ricci tensor of these frozen cones.  Here ``frozen'' means that the transverse parameter is treated as a constant; later, at a given radial scale, the slowly varying metric will be compared with the cone obtained by freezing the value of that parameter and ignoring its radial derivatives.  In a local holomorphic Hopf trivialization, write $\Phi=e^F$, put $D=F_s^2+F_{ss}$, and let $\Sigma$ be the horizontal Schur complement of the Hermitian matrix of $\sqrt{-1}\partial\bar\partial e^F$.  The determinant factors as $\det(\sqrt{-1}\partial\bar\partial e^F)=e^{2F}D\det\Sigma$ up to the harmless holomorphic Jacobian of the fiber coordinate.  When $F=bs+t\psi$, one has $D=b^2$, and the quadratic $t^2|\partial\psi|^2$ terms cancel in the Schur complement, leaving $\Sigma=\omega_{b,t}$.  Applying $-\sqrt{-1}\partial\bar\partial\log\det$ gives $\operatorname{Ric}(\Omega_{b,t})=\pi^*(\operatorname{Ric}(\omega_{b,t})-2\omega_{b,t})$.  The radial--Reeb complex line therefore has zero Ricci curvature, while the horizontal complex line has the uniform positive margin above.

If $\eta_t$ is the normalized contact form on the link of $\Omega_{a,t}$, then $d\eta_t$ is a fixed multiple of $\pi^*\omega_{a,t}$ and the period of $\eta_t$ on a Hopf fiber is independent of $t$.  Fiber integration gives
\begin{equation}\label{eq:appendix-contact-volume}
\int_{S^3}\eta_t\wedge d\eta_t
 =c\int_{\mathbf{CP}^1}\omega_{a,t}
 =ca\int_{\mathbf{CP}^1}\omega_0.
\end{equation}
The link volume is therefore independent of $t$ because $[\omega_{a,t}]=a[\omega_0]$.  This fixed-volume property will play the same role as $\operatorname{tr}_h\dot h=0$ in the Riemannian model, namely it prevents a first-order volume defect from appearing in the radial Ricci calculation.

Note that the cones are nonisometric.  Let $S_t$ denote the scalar curvature of $\omega_{a,t}$ and let $\overline S$ be its average, which is independent of $t$ because the K\"ahler class is fixed.  Our choice of $\psi$ gives $\dot S_0\neq0$, and hence the scalar-curvature variance $\mathcal V(t)=\int_{\mathbf{CP}^1}(S_t-\overline S)^2\omega_{a,t}$ satisfies $\mathcal V(t)=c_0t^2+O(t^3)$ with $c_0>0$.  After decreasing $\varepsilon$, $\mathcal V$ is strictly increasing on $(0,\varepsilon)$, so the cones $\Omega_{a,t}$ are pairwise nonisometric there.  The underlying affine variety $\mathbb C^2$, the holomorphic dilation, the Hopf/Reeb action, and the contact volume remain fixed throughout this family.

\subsection{Slow modulation and the radial Ricci buffer}
\label{subsec:appendix-modulation}

Choose a smooth periodic function $T:\mathbb R\to[0,\varepsilon]$ whose image is all of $[0,\varepsilon]$ and which is constant on a nonempty interval at each endpoint.  Write $\sigma=-s$.  On a sufficiently small punctured ball set
\begin{equation}\label{eq:appendix-main-potential}
F(s,[z])=as+K\log\sigma+T(\log\sigma)\psi([z]),
\qquad
\Omega=\sqrt{-1}\partial\bar\partial e^F.
\end{equation}
Now we explain what the three terms do.
\begin{itemize}
\item The term $as$ gives the homogeneous cone scale $|z|^{2a}$.
\item The term $T(\log\sigma)\psi$ moves through the family constructed above.  If $\tau(s)=T(\log\sigma)$, then $\tau^{(k)}(s)=O(\sigma^{-k})$.  Thus the phase changes infinitely often as $s\to-\infty$, but becomes asymptotically constant on every bounded window in the translated logarithmic variable.
\item The term $K\log\sigma$ supplies radial Ricci curvature.  Its first derivative is $-K/\sigma\to0$, so it does not change a tangent cone, while its second derivative is $-K/\sigma^2$.  In the Ricci form this produces the favorable leading term $2K\sigma^{-2}$ in the radial direction.
\end{itemize}
We now verify that the last term supplies enough radial positivity without disturbing the frozen transverse geometry.

\begin{lemma}[Slow modulation]
\label{lem:appendix-slow-modulation}
After fixing the small family above, one can choose $K$ and then $\sigma_0$ such that $\Omega$ is K\"ahler on $\{\sigma>\sigma_0\}$ and $\operatorname{Ric}(\Omega)\geq0$ there.
\end{lemma}

\begin{proof}
We use the determinant factorization from the frozen calculation.  The slowly varying metric is compared at each $s$ with the frozen cone $\Omega_{a,\tau(s)}$.  In a horizontal/radial complex frame, differentiating $-2F-\log D-\log\det\Sigma$ gives
\begin{equation}\label{eq:appendix-ricci-block}
\operatorname{Ric}(\Omega)=
\begin{pmatrix}
H_{\tau}+E_H&E_M\\
E_M^*&2K\sigma^{-2}+E_R
\end{pmatrix},
\qquad
H_{\tau}=\operatorname{Ric}(\omega_{a,\tau})-2\omega_{a,\tau}.
\end{equation}
The terms in this matrix have the following sizes.
\begin{itemize}
\item The frozen horizontal block satisfies $H_\tau\geq4\delta\omega_0$.
\item The drift of the horizontal coefficients gives $|E_H|\leq C_K\sigma^{-1}$.  This term may depend on $K$, but it is absorbed by taking $\sigma_0$ large after $K$ has been fixed.
\item The leading mixed term comes from the angular derivative of $\tau(s)\psi$ and satisfies $|E_M|\leq C_1\|\psi\|_{C^3}\sigma^{-1}+C_K\sigma^{-2}$.  The $K\log\sigma$ term is independent of the base variable, so it does not produce a leading $K\sigma^{-1}$ mixed term.
\item After the explicit buffer $2K\sigma^{-2}$ has been separated, the radial remainder satisfies $E_R\geq-C_2(\psi,T)\sigma^{-2}-C_K\sigma^{-3}$.  The terms of order $\sigma^{-2}$ come from $\tau''$ and $(\tau')^2$; derivatives involving the radial $K/\sigma$ perturbation enter one order lower.
\end{itemize}
For $\sigma$ sufficiently large, the horizontal block is at least $2\delta\omega_0$.  The Schur-complement loss from the mixed block is therefore at most $C_3\|\psi\|_{C^3}^2\sigma^{-2}+C_K\sigma^{-3}$.  The constants multiplying the leading $\sigma^{-2}$ loss are independent of $K$.  We first choose $K$ so that $2K>C_2(\psi,T)+C_3\|\psi\|_{C^3}^2+1$ and then choose $\sigma_0\gg K$ so that every $K$-dependent lower-order term and the horizontal error are absorbed.  The radial Schur complement is then nonnegative.  The same comparison with the positive frozen cones shows that $\Omega>0$ for $\sigma>\sigma_0$.
\end{proof}

The completion of this punctured metric adds a single point $p$.  Indeed $e^F\asymp e^{as}\sigma^K$, and the exponential decay makes both the radial length to $s=-\infty$ and the diameters of the level hypersurfaces tend to zero.

Let $s_j\to-\infty$ and assume $T(\log(-s_j))\to t_*$.  Set $D_j(w)=e^{s_j/2}w$ and $\lambda_j=e^{-as_j}(-s_j)^{-K}$, so the rescaled metric is $\lambda_jD_j^*\Omega$.  On a compact annulus write $\ell=\log|w|^2$.  Then $K\log(1-\ell/(-s_j))\to0$ and $T(\log(-s_j-\ell))\to t_*$ in every $C^k$ norm.  It follows that $\lambda_jD_j^*\Omega$ converges smoothly away from the vertex to $\Omega_{a,t_*}=\sqrt{-1}\partial\bar\partial\exp(a\log|w|^2+t_*\psi([w]))$.  The shrinking diameter of the innermost region upgrades this to pointed Gromov--Hausdorff convergence including the vertex.  Since the cluster set of $T$ is $[0,\varepsilon]$, the tangent-cone set contains the pairwise nonisometric family constructed in Subsection~\ref{subsec:appendix-hopf}.

\subsection{The cap and smooth Gromov--Hausdorff approximants}
\label{subsec:appendix-cap}

Now we seek to realize the completed metric as a Gromov--Hausdorff limit of smooth K\"ahler metrics.  A flat center causes no difficulty by itself.  The issue is joining it to the outer metric at a radius which tends to zero while keeping one Ricci lower bound.  If a fixed transition profile has $\operatorname{Ric}\geq-Cg$ and is shrunk by a factor $\rho$, the corresponding lower-bound constant is of order $C\rho^{-2}$.

Choose $S_j\to-\infty$ in the interior of flat intervals on which $T\equiv0$.  On a collar of $S_j$ the outer metric is radial, with potential $P_*(s)=e^{as}(-s)^K$.  For a radial potential $P=P(s)$ on $\mathbb C^2\setminus\{0\}$, one has $\det g_P=c e^{-2s}P'P''$.  It is convenient to set $G_P(s)=2s-\log(P'P'')$, so $\operatorname{Ric}(\sqrt{-1}\partial\bar\partial P)=\sqrt{-1}\partial\bar\partial G_P$ and $(P'^2)'=2e^{2s-G_P}$.  For $P_*$, direct differentiation gives $G_*'(s)=2(1-a)+2K/(-s)+O(s^{-2})$ and $G_*''(s)=2K/s^2+O(|s|^{-3})$, so both derivatives are positive sufficiently far down the end.

\begin{lemma}[Radial cap]
\label{lem:appendix-cap}
There are smooth strictly plurisubharmonic radial potentials $P_j$ and a constant $C$ independent of $j$ such that $P_j=P_*$ up to an additive constant on a collar of $S_j$, the metric $\sqrt{-1}\partial\bar\partial P_j$ is a positive multiple of the Euclidean metric near the origin, $\operatorname{Ric}(\sqrt{-1}\partial\bar\partial P_j)\geq-C\sqrt{-1}\partial\bar\partial P_j$, and the diameter of the capped region tends to zero.
\end{lemma}

\begin{proof}
Set $A_j=2S_j$.  We first extend the Ricci potential rather than the K\"ahler potential.  Choose a smooth function $\bar G_j$ on $(-\infty,S_j]$ which is constant for $s\leq A_j-1$, equals $G_*$ on $[A_j+1,S_j]$, and satisfies $\bar G_j'\geq0$ and $\bar G_j''\geq0$ everywhere.  Such an extension is obtained by choosing $\bar G_j'$ to be a smooth nonnegative nondecreasing function which is zero on the far left and agrees with $G_*'$ near $A_j+1$; the positivity of $G_*'$ and $G_*''$ for large $j$ makes the interpolation possible.

Recovering $P$ from $\bar G_j$ with the smooth-center condition would not in general give exactly the required value of $P_*'(S_j)$.  The missing mass comes only from $s\leq A_j$ and is $O(e^{2aA_j}|A_j|^{2K})$.  Choose a fixed bump $\chi_j$ supported in $(S_j-3,S_j-2)$, where $\bar G_j=G_*$, and put $G_j=\bar G_j+\alpha_j\chi_j$.  A unit change of $\alpha_j$ changes the matching integral at scale $e^{2aS_j}|S_j|^{2K}$.  The intermediate value theorem therefore gives $|\alpha_j|=O(e^{2aS_j})$ and allows us to impose
\begin{equation}\label{eq:appendix-cap-matching}
2\int_{-\infty}^{S_j}e^{2s-G_j(s)}\,ds=P_*'(S_j)^2,
\qquad
P_j'(s)=\left(2\int_{-\infty}^{s}e^{2t-G_j(t)}\,dt\right)^{1/2}.
\end{equation}
Choose the additive constant in $P_j$ so that $P_j=P_*$ at the end of the matching collar.  Since $G_j=G_*$ on that collar and the first equality in \eqref{eq:appendix-cap-matching} matches $P'$, the first-order equation for $P'^2$ shows that the two radial metrics agree throughout the collar.

Where $G_j$ is constant, the second formula in \eqref{eq:appendix-cap-matching} gives $P_j'=c_je^s$, hence $P_j=c_je^s+\mathrm{const}$, and the cap is exactly Euclidean up to scale near the center.  Before adding $\alpha_j\chi_j$ we have $\bar G_j',\bar G_j''\geq0$ everywhere, so the only possible negative Ricci contribution comes from the tiny correcting bump.  On its support the negative parts of $G_j'$ and $G_j''$ are $O(|\alpha_j|)$, while $P_j'$ and $P_j''$ are comparable to $e^{aS_j}|S_j|^K$.  Thus the ratios of the negative Ricci eigenvalues to the corresponding metric eigenvalues are $O(e^{aS_j}|S_j|^{-K})\to0$, which gives a uniform lower Ricci bound.  The radial and angular sizes of the cap are $O(e^{aS_j/2}|S_j|^{K/2})$, so its diameter tends to zero.
\end{proof}

Replace the region $\{s<S_j\}$ by the cap, leaving the common outer metric unchanged.  The deleted region in the singular metric and the inserted smooth cap both have diameter tending to zero.  Identifying the common exterior and collapsing the cap to the completion point therefore gives local Gromov--Hausdorff convergence.  The oscillating region has nonnegative Ricci by Lemma~\ref{lem:appendix-slow-modulation}, and the caps have the uniform lower bound from Lemma~\ref{lem:appendix-cap}.

\subsection{Compactification and polarization}
\label{subsec:appendix-compactification}

We now place the local construction in a fixed projective surface.  Choose an affine coordinate chart $U\simeq B_{2r_2}(0)\subset\mathbb C^2$ centered at a point $p\in\mathbf{CP}^2$, and choose $0<r_1<r_2$ so small that the oscillating potential is strictly plurisubharmonic on $B_{2r_2}\setminus\{0\}$ and the entire cap construction takes place in $B_{r_1}$ for all large $j$.  Let $\phi_{\mathrm{FS}}$ be a Fubini--Study potential on $U$ and let $\phi_{\mathrm{in}}=e^F$ denote the common singular potential outside the caps.

The different radial growth rates make a fixed potential gluing possible.  Since $a<1$ and $\psi$ is bounded, $\phi_{\mathrm{in}}/|z|^2\to\infty$ as $z\to0$.  We may therefore choose radii $0<r_0<r_1<r_2$, a fixed constant $\lambda>0$, and additive constants so that $\lambda\phi_{\mathrm{in}}$ lies strictly above $\phi_{\mathrm{FS}}$ on a neighborhood of $\partial B_{r_0}$ and strictly below it on a neighborhood of $\partial B_{r_2}$.  Applying a regularized maximum on the fixed annulus $B_{r_2}\setminus\overline{B}_{r_0}$ produces a smooth strictly plurisubharmonic function which agrees with $\lambda\phi_{\mathrm{in}}$ near $\partial B_{r_0}$ and with $\phi_{\mathrm{FS}}$ near $\partial B_{r_2}$.  This interpolation is chosen once and is independent of $j$.

Scale the local caps by the same fixed factor $\lambda$ and use them inside $B_{r_0}$.  The resulting local metric $\widehat\omega_j$ agrees with the common oscillating metric on the annulus outside the cap, with the fixed regularized-maximum metric farther out, and with $\omega_{\mathrm{FS}}$ near $\partial B_{r_2}$.  We can therefore define a global form by taking $\omega_j=\widehat\omega_j$ on $B_{r_2}$ and $\omega_j=\omega_{\mathrm{FS}}$ on $\mathbf{CP}^2\setminus B_{r_2}$.  Since the two forms agree on a collar of the boundary, this gives a smooth global K\"ahler metric.

The polarization is unchanged.  On the coordinate ball, $\omega_j-\omega_{\mathrm{FS}}=\sqrt{-1}\partial\bar\partial u_j$ for a smooth function $u_j$ which vanishes near $\partial B_{r_2}$; extending $u_j$ by zero gives a global function on $\mathbf{CP}^2$.  Hence $[\omega_j]=[\omega_{\mathrm{FS}}]=c_1(\mathcal O_{\mathbf{CP}^2}(1))$.  A fixed integral rescaling gives the same construction in $c_1(\mathcal O_{\mathbf{CP}^2}(k))$ if desired.

The Ricci lower bound is uniform on the compactification.  The deep oscillating region has nonnegative Ricci, the caps have the lower bound from Lemma~\ref{lem:appendix-cap}, and the regularized-maximum annulus is one fixed smooth K\"ahler metric on a fixed compact set, so its Ricci tensor has one fixed lower bound.  Outside the chart the metric is Fubini--Study.  The global Gromov--Hausdorff convergence follows from the local convergence because the metrics are identical outside the shrinking caps.  For noncollapse, choose a fixed annulus sufficiently close to the completed point that its distance from $p$ is less than $1$ in the limiting local metric.  For all large $j$ the same annulus lies in $B_{\omega_j}(p,1)$, its metric is independent of $j$, and it has positive fixed volume.  After discarding finitely many terms this gives the constant $v$ in Theorem~\ref{thm:lower-ricci-nonunique}.

The tangent cones at $p$ are unaffected by the fixed outer interpolation.  By Subsection~\ref{subsec:appendix-modulation}, every $t\in[0,\varepsilon]$ occurs as the frozen phase of a tangent-cone subsequence, and Subsection~\ref{subsec:appendix-hopf} shows that these cones are pairwise nonisometric.  This proves Theorem~\ref{thm:lower-ricci-nonunique}.

\begin{remark}
The example is compatible with the lower-Ricci algebraic structure used throughout the paper.  All of the tangent cones above have underlying affine variety $\mathbb C^2$ and the same Hopf/Reeb action, while their transverse K\"ahler metrics differ.  Thus the uniqueness conclusion of Theorem~\ref{thm:main-log-ds} cannot be recovered from polarization, noncollapse, and a lower Ricci bound alone.
\end{remark}


\begin{thebibliography}{99}

\bibitem{AndersonRicci}
M.~T. Anderson,
\newblock Convergence and rigidity of manifolds under Ricci curvature bounds,
\newblock \emph{Invent. Math.} \textbf{102} (1990), no.~2, 429--445.

\bibitem{BedfordTaylorCapacity}
E.~Bedford and B.~A. Taylor,
\newblock A new capacity for plurisubharmonic functions,
\newblock \emph{Acta Math.} \textbf{149} (1982), no.~1--2, 1--40.

\bibitem{BirkesOrbits}
D.~Birkes,
\newblock Orbits of linear algebraic groups,
\newblock \emph{Ann. of Math. (2)} \textbf{93} (1971), no.~3, 459--475.

\bibitem{CheegerColdingI}
J.~Cheeger and T.~H. Colding,
\newblock On the structure of spaces with Ricci curvature bounded below. I,
\newblock \emph{J. Differential Geom.} \textbf{46} (1997), no.~3, 406--480.

\bibitem{CheegerColdingII}
J.~Cheeger and T.~H. Colding,
\newblock On the structure of spaces with Ricci curvature bounded below. II,
\newblock \emph{J. Differential Geom.} \textbf{54} (2000), no.~1, 13--35.

\bibitem{CheegerColdingIII}
J.~Cheeger and T.~H. Colding,
\newblock On the structure of spaces with Ricci curvature bounded below. III,
\newblock \emph{J. Differential Geom.} \textbf{54} (2000), no.~1, 37--74.

\bibitem{CheegerJiangNaber}
J.~Cheeger, W.~Jiang, and A.~Naber,
\newblock Rectifiability of singular sets of noncollapsed limit spaces with
Ricci curvature bounded below,
\newblock \emph{Ann. of Math. (2)} \textbf{193} (2021), no.~2, 407--538.

\bibitem{CheegerNaberQuant}
J.~Cheeger and A.~Naber,
\newblock Lower bounds on Ricci curvature and quantitative behavior of singular sets,
\newblock \emph{Invent. Math.} \textbf{191} (2013), no.~2, 321--339.

\bibitem{CDSI}
X.-X.~Chen, S.~Donaldson, and S.~Sun,
\newblock K\"ahler--Einstein metrics on Fano manifolds. I:
Approximation of metrics with cone singularities,
\newblock \emph{J. Amer. Math. Soc.} \textbf{28} (2015), no.~1, 183--197.

\bibitem{CDSII}
X.-X.~Chen, S.~Donaldson, and S.~Sun,
\newblock K\"ahler--Einstein metrics on Fano manifolds. II:
Limits with cone angle less than $2\pi$,
\newblock \emph{J. Amer. Math. Soc.} \textbf{28} (2015), no.~1, 199--234.

\bibitem{CDSIII}
X.-X.~Chen, S.~Donaldson, and S.~Sun,
\newblock K\"ahler--Einstein metrics on Fano manifolds. III:
Limits as cone angle approaches $2\pi$ and completion of the main proof,
\newblock \emph{J. Amer. Math. Soc.} \textbf{28} (2015), no.~1, 235--278.

\bibitem{ColdingNaberTangents}
T.~H. Colding and A.~Naber,
\newblock Characterization of tangent cones of noncollapsed limits with
lower Ricci bounds and applications,
\newblock \emph{Geom. Funct. Anal.} \textbf{23} (2013), no.~1, 134--148.

\bibitem{CollinsSzekelyhidi}
T.~C. Collins and G.~Sz\'ekelyhidi,
\newblock Sasaki--Einstein metrics and $K$-stability,
\newblock \emph{Geom. Topol.} \textbf{23} (2019), no.~3, 1339--1413.

\bibitem{DatarSzekelyhidi}
V.~Datar and G.~Sz\'ekelyhidi,
\newblock K\"ahler--Einstein metrics along the smooth continuity method,
\newblock \emph{Geom. Funct. Anal.} \textbf{26} (2016), 975--1010.

\bibitem{DSI}
S.~Donaldson and S.~Sun,
\newblock Gromov--Hausdorff limits of K\"ahler manifolds and algebraic geometry,
\newblock \emph{Acta Math.} \textbf{213} (2014), no.~1, 63--106.

\bibitem{DSII}
S.~Donaldson and S.~Sun,
\newblock Gromov--Hausdorff limits of K\"ahler manifolds and algebraic geometry. II,
\newblock \emph{J. Differential Geom.} \textbf{107} (2017), no.~2, 327--371.

\bibitem{FGSW}
X.~Fu, B.~Guo, J.~Song, and J.~Wang,
\newblock Fundamental groups of compact K\"ahler varieties with nef anti canonical bundle,
\newblock preprint, arXiv:2602.07420 (2026).

\bibitem{GuoPhongSongSturmSobolev}
B.~Guo, D.~H. Phong, J.~Song, and J.~Sturm,
\newblock Sobolev inequalities on K\"ahler spaces,
\newblock preprint, arXiv:2311.00221 (2023).

\bibitem{HaimanSturmfels}
M.~Haiman and B.~Sturmfels,
\newblock Multigraded Hilbert schemes,
\newblock \emph{J. Algebraic Geom.} \textbf{13} (2004), no.~4, 725--769.

\bibitem{HeSunFrankel}
W.~He and S.~Sun,
\newblock Frankel conjecture and Sasaki geometry,
\newblock \emph{Adv. Math.} \textbf{291} (2016), 912--960.

\bibitem{KollarRationalCurves}
J.~Koll\'ar,
\newblock \emph{Rational Curves on Algebraic Varieties},
\newblock Ergebnisse der Mathematik und ihrer Grenzgebiete, 3.~Folge,
vol.~32, Springer-Verlag, Berlin, 1996.

\bibitem{KollarHullsHusks}
J.~Koll\'ar,
\newblock Hulls and husks,
\newblock preprint, arXiv:0805.0576 (2008).

\bibitem{KollarFamiliesDivisors}
J.~Koll\'ar,
\newblock Families of divisors,
\newblock preprint, arXiv:1910.00937 (2019).

\bibitem{KollarFamiliesGeneralType}
J.~Koll\'ar,
\newblock \emph{Families of Varieties of General Type},
\newblock Cambridge Tracts in Mathematics, vol.~231,
Cambridge University Press, Cambridge, 2023.

\bibitem{ChiLiWeightedSolitons}
C.~Li,
\newblock Notes on weighted K\"ahler--Ricci solitons and application to
Ricci-flat K\"ahler cone metrics,
\newblock preprint, arXiv:2107.02088 (2021).

\bibitem{LLXGuidedTour}
C.~Li, Y.~Liu, and C.~Xu,
\newblock A guided tour to normalized volume,
\newblock in \emph{Geometric Analysis: In Honor of Gang Tian's 60th Birthday},
Progress in Mathematics, vol.~333,
Birkh\"auser, Cham, 2020, 167--219.

\bibitem{LiWangXu}
C.~Li, X.~Wang, and C.~Xu,
\newblock Algebraicity of the metric tangent cones and equivariant $K$-stability,
\newblock \emph{J. Amer. Math. Soc.} \textbf{34} (2021), no.~4, 1175--1214.

\bibitem{LiXuValuations}
C.~Li and C.~Xu,
\newblock Stability of valuations: higher rational rank,
\newblock \emph{Peking Math. J.} \textbf{1} (2018), no.~1, 1--79.

\bibitem{SLI}
G.~Liu and G.~Sz\'ekelyhidi,
\newblock Gromov--Hausdorff limits of K\"ahler manifolds with Ricci curvature bounded below,
\newblock \emph{Geom. Funct. Anal.} \textbf{32} (2022), no.~2, 236--279.

\bibitem{SLII}
G.~Liu and G.~Sz\'ekelyhidi,
\newblock Gromov--Hausdorff limits of K\"ahler manifolds with Ricci curvature bounded below. II,
\newblock \emph{Comm. Pure Appl. Math.} \textbf{74} (2021), no.~5, 909--931.

\bibitem{MartelliSparksYau}
D.~Martelli, J.~Sparks, and S.-T.~Yau,
\newblock Sasaki--Einstein manifolds and volume minimisation,
\newblock \emph{Comm. Math. Phys.} \textbf{280} (2008), no.~3, 611--673.

\bibitem{PerelmanNonunique}
G.~Perelman,
\newblock A complete Riemannian manifold of positive Ricci curvature with
Euclidean volume growth and nonunique asymptotic cone,
\newblock in K.~Grove and P.~Petersen (eds.),
\emph{Comparison Geometry},
Math. Sci. Res. Inst. Publ., vol.~30,
Cambridge University Press, Cambridge, 1997, 165--166.

\bibitem{SzekelyhidiRCD}
G.~Sz\'ekelyhidi,
\newblock Singular K\"ahler--Einstein metrics and RCD spaces,
\newblock \emph{Forum Math. Pi} \textbf{13} (2025), e24.

\bibitem{WangZhuBakryEmery}
F.~Wang and X.~Zhu,
\newblock The structure of spaces with Bakry--\'Emery Ricci curvature bounded below,
\newblock \emph{J. Reine Angew. Math.} \textbf{757} (2019), 1--50.

\end{thebibliography}
\end{document}